\documentclass[a4paper]{amsart}

\usepackage[utf8]{inputenc}
\usepackage[english]{babel}
\usepackage{tikz-cd}
\usetikzlibrary{positioning,decorations.text,quotes}
\usepackage{amsmath,amsfonts,amssymb}
\usepackage{mathtools}
\usepackage{float}
\usepackage{amsthm}
\usepackage{hyperref}
\usepackage{stackrel,amssymb,amsmath}
\usepackage{pgfplots}
\usepackage{graphicx}
\usepackage{xcolor}
\usepackage{bbm}
\usepackage{amssymb}
\usepackage{amsmath,amsthm,geometry,enumerate}
\usepackage{graphicx}
\usepackage{psfrag}
\usepackage{amscd}
\usepackage{comment}
\usepackage{hyperref}
\usepackage[all]{xy}
\usepackage{enumerate}
\usepackage{tikz-cd}
\usepackage{booktabs}
\usepackage{color}
\usepackage{todonotes}

\usepackage{braket}

\newcommand{\ZZ}{\mathbb{Z}}

\newcommand{\RR}{\mathbb{R}}

\newcommand{\A}{\mathcal{A}}
\newcommand{\C}{\mathcal{C}}
\renewcommand{\O}{\mathcal{O}}

\renewcommand{\L}{\mathcal{L}}

\newcommand{\F}{\mathcal{F}}
\newcommand{\D}{\mathcal{D}}
\newcommand{\X}{\mathcal{X}}

\newcommand{\J}{\mathcal{J}}
\newcommand{\M}{\mathcal{M}}
\renewcommand{\S}{\mathcal{S}}
\renewcommand{\P}{\mathcal{P}}

\newcommand{\m}{\mathfrak{m}}
\newcommand{\mEhA}{\mathfrak{m}_{(e,h,A)}}
\newcommand{\n}{\mathfrak{n}}
\newcommand{\nZ}{\mathfrak{n}_Z}

\newcommand{\Mbargn}{\overline{\mathcal{M}}_{g,n}}
\newcommand{\Ggn}{G_{g,n}}

\pgfplotsset{compat=1.17}

\newcommand{\Address}{{
  \bigskip
  \footnotesize

  M.~Fava, \textsc{Department of Mathematical Sciences, University of Liverpool,
    Liverpool, L69 7ZL, United Kingdom}\par\nopagebreak
  \textit{E-mail address}: \texttt{marco.fava@liverpool.ac.uk}
  }}
  
\newtheorem{theorem}{Theorem}[section]
\newtheorem{corollary}[theorem]{Corollary}
\newtheorem{lemma}[theorem]{Lemma}
\newtheorem{proposition}[theorem]{Proposition}
\newtheorem{proposition-definition}[theorem]{Proposition-Definition}
\newtheorem{lemma-definition}[theorem]{Lemma-Definition}

\theoremstyle{definition}
\newtheorem{definition}[theorem]{Definition}
\newtheorem{example}[theorem]{Example}
\newtheorem{remark}[theorem]{Remark}

\numberwithin{equation}{section}

\newtheorem{manualtheoreminner}{Question}
\newenvironment{manualtheorem}[1]{%
  \renewcommand\themanualtheoreminner{#1}%
  \manualtheoreminner
}{\endmanualtheoreminner}

\title{On a combinatorial classification of fine compactified universal Jacobians}

\begin{document}

\author{Marco Fava}

	\begin{abstract}
		Extending the definition of $V$-stability conditions, given by Viviani in the recent preprint \cite{viviani2023new}, we introduce the notion of \emph{universal stability conditions}. Building on results by Pagani and Tommasi in \cite{pagani2023stability}, we show that fine compactified universal Jacobians, that is, fine compactified Jacobians over the moduli spaces of stable pointed curves $\Mbargn$, are combinatorially classified by universal stability conditions. 

        We use these stability conditions to show the following. The inclusion of fine compactified universal Jacobians of type $(g,n)$ whose fibres over geometric points are classical, that is, they are constructed by some numerical polarisation, into the class of all fine compactified universal Jacobians, is strict, in general, for any $g\geq 2$. This answers a question of Pagani and Tommasi.
	\end{abstract} 
	
	\maketitle
	
	\tableofcontents
	
	\section{Introduction}
 
	In this paper, we extensively use results from the recent preprints \cite{pagani2023stability} and \cite{viviani2023new}, by Pagani-Tommasi and Viviani respectively, to give a complete classification of fine compactified universal (over each moduli stack $\overline{\mathcal{M}}_{g,n}$ of stable $n$-pointed curves of genus $g$) Jacobians in terms of universal $V$-stability conditions. This addresses \cite[Open Question~(4), p.7]{viviani2023new}. In doing so, we answer, for all $g\geq 2$, the following (raised in  \cite[Section $10$]{pagani2023stability}):

    \begin{manualtheorem}{$\star$}\label{questintro}
    Are there fine compactified universal  Jacobians whose fibres over some geometric points are not classical?
    \end{manualtheorem}

    We answer this question by giving, for each $g \geq 2$, the minimal natural number $n$ such that at least one such fine compactified universal Jacobian exists over $\overline{\mathcal{M}}_{g,n}$.

	\vspace{0.5cm}

 {\renewcommand*{\thetheorem}{\Alph{theorem}}
 \setcounter{theorem}{0}
	
 Let $(g,n)$ be a hyperbolic pair, that is $g,n\in\ZZ_{\geq0}$ such that $2g-2+n>0$.
 By a \emph{degree $d$ fine compactified universal Jacobian of type $(g,n)$} we mean an open and proper substack of the moduli stack of rank~$1$ torsion free sheaves for the universal family $\overline{\C}_{g,n}/\Mbargn$.
	
	Compactified Jacobians have been constructed via different methods by Oda-Seshadri \cite{Oda1979CompactificationsOT}, for a single nodal curve, by Altman-Kleiman \cite{ALTMAN198050}, Esteves \cite{esteves}, Simpson \cite{simpson} for families of curves over a scheme, and by  Caporaso \cite{caporaso}, Pandharipande \cite{pandharipande1995compactification} for the universal family.
	Following \cite{viviani2023new}, we will call these \emph{classical fine compactified Jacobians} (see Definitions~\ref{classicaljacdef} and~\ref{fclassicalcj}-~\ref{phitocdef} for a precise definition). 
 
 The first construction of space of stability conditions for classical fine compactified universal Jacobians was given in \cite{Kass_2019}. Each of the maximal dimensional chambers of the stability space defined in loc.cit. corresponds to a fine compactified universal Jacobian, that we will refer to as \emph{classical fine compactified universal Jacobians} (Definition~\ref{classicaljacdef}). (These  were earlier constructed in \cite{melo2016compactifications} and \cite{Kass_2017}). The existence of non-classical fine compactified universal Jacobians was first observed in \cite{PTgenus1}, and the examples were given when $g=1$ and $n \geq 6$. However, to the best of our knowledge, all known examples of fine compactified universal Jacobians have fibres over each geometric point of $\overline{\mathcal{M}}_{g,n}$ that are classical compactified Jacobians (in the sense of the previous paragraph). We call such compactified universal Jacobians \emph{semi-classical} (Definition~\ref{classicaljacdef}).
 
 Recently, Pagani and Tommasi gave in \cite{pagani2023stability} a complete combinatorial classification of all fine compactified universal Jacobians. This classification is fairly complicated as it uses a large amount of data. For example trying to approach Question~\ref{questintro} by means of this combinatorial description seems very hard.

	Inspired by the more recent paper of Viviani \cite{viviani2023new}, in this work we introduce the notion of \emph{universal stability conditions} (Definition \ref{cstabuniv}), whose  restriction to a curve $X$ (a $c$-stability condition, Definition~\ref{cstabGamma}) agrees with the notion of a $V$-stability condition introduced in \cite{viviani2023new}. We show that the set of universal stability conditions is precisely the combinatorial counterpart of the set of fine compactified universal Jacobians. It is by means of this explicit combinatorial description that we are able to fully answer Question \ref{questintro}. 
	\vspace{0.5cm}

    Given a degree $d$ universal stability condition $\m$ of type $(g,n)$, we define its associated fine compactified universal Jacobian $\overline{\J}_{g,n}(\m)$ by taking the set of \emph{$\m$-stable} families of sheaves in $\textup{Simp}(\overline{\C}_{g,n}/\Mbargn)$ (Definition~\ref{stablesheavesdef}). 
    
    Viceversa, let $\overline{\J}_{g,n}$ be a degree $d$ fine compactified universal Jacobian.     
    We can obtain its \emph{associated universal stability condition} $\m(\overline{\J}_{g,n})$ (Definition~\ref{jtomdef}) by taking, for every components of all the marked \emph{vine curves} (i.e. curves with two irreducible components and no self intersections) $[X,\{p_1,\ldots,p_n\}]\in\Mbargn$, the minimum degrees of sheaves $F$ such that  $[X,\{p_1,\ldots,p_n\},F]\in\overline{\J}_{g,n}$.

    These assignments give the following, which shows that fine compactified universal Jacobians are, in fact, fully classified by universal stability conditions.
 
 \begin{theorem}\label{mainthm}
    
    The map
		\begin{align*}
		\begin{Bmatrix}
			\textup{Degree $d$ fine compactified}\\ \textup{universal Jacobians of type $(g,n)$}
		\end{Bmatrix}&\to \begin{Bmatrix}
			\textup{Degree $d$ universal}\\ \textup{stability conditions of type $(g,n)$}
		\end{Bmatrix}\\
			\overline{\J}_{g,n}&\mapsto\m(\overline{\J}_{g,n}), 
		\end{align*}
	is a bijection with inverse $\m\mapsto \overline{\J}_{g,n}(\m)$.
    \end{theorem}
	
	\vspace{0.5cm}
	
    We observe that a universal stability condition can be described as the limit of a collection of $c$-stability conditions that are \emph{compatible with $\Ggn$-morphisms} (Definition~\ref{compatiblewithmorphs}) and agree on their common domains. 
 
    Similarly to the case of $V$-stability conditions on a graph $\Gamma$, a $\Ggn$-morphisms compatible numerical polarisation $\phi^\Gamma$ induces a classical $c$-stability condition $\m^\Gamma(\phi^\Gamma)$ (Definition~\ref{phitocdef}).

    However, not all universal stability conditions arise as limits of classical $c$-stability conditions.
    In fact, in Section~\ref{seccomparison}, we find suitable non classical $c$-stability conditions on trivalent graphs in $\Ggn$, which are restriction of a universal stability condition. Furthermore, we show that, given any genus $g$, such graphs have the minimal number of markings $n$ for this to be possible. Hence, we deduce our main result,  answering Question~\ref{questintro}:
	\begin{theorem} \label{theoremintro} Let $(g,n)$ be a hyperbolic pair.
		\begin{enumerate}
			\item[1.] The inclusion
			\begin{equation*}
				\begin{Bmatrix}
				\textup{Degree $d$ semi-classical fine compactified}\\ 
                    \textup{universal Jacobians of type $(g,n)$}
				\end{Bmatrix}
				\hookrightarrow
				\begin{Bmatrix}
					\textup{Degree $d$ fine compactified}\\ \textup{universal Jacobians of type $(g,n)$}
				\end{Bmatrix}
			\end{equation*}
			is a bijection if and only if any of the following occurs:
			\begin{itemize}
                \item $g=0$ (straightforward).
				\item $n=0$ (shown in \cite[Theorem~9.6]{pagani2023stability}, see also Example~\ref{exg1}).
				\item $g=1$ (shown in  \cite[Proposition~3.15]{PTgenus1}).
				\item $g=2$ and $n\leq5$ (Example \ref{ex26}, Remark~\ref{rhys25}).
				\item $g=3$ and $n=1$ (Example \ref{ex32}, Proposition \ref{prop31}).
			\end{itemize}
			\item[2.] There are a finite number of equivalence classes of degree $d$ fine compactified universal Jacobians over $\overline{\C}_{g,n}\to\Mbargn$ up to the translation action by the universal, relative degree~$0$ line bundles group $\textup{PicRel}^0_{g,n}(\ZZ)$.
		\end{enumerate}
		
	\end{theorem}
	
	More specifically, given any fine compactified universal Jacobian , we can decide whether it is classical or semi-classical by analysing the system of inequalities induced by its associated universal stability condition, via Proposition~\ref{phitocprop}.
	
    We can also ask the related question: what are the pairs $(g,n)$ such that the set of classical fine compactified universal Jacobians is strictly included in the set of all fine compactified universal Jacobians? (The answer to this  was already known for $g\leq1$ by work of \cite{PTgenus1} and for $n=0$ by work of \cite{pagani2023stability}). 
    
    The $n$ given in Theorem~\ref{theoremintro} is sharp for the latter inclusion to be strict whenever $g \geq 3$. Indeed, by Proposition~\ref{prop31}, every fine compactified universal Jacobian over $\overline{\M}_{3,1}$ is classical. In  Corollary~\ref{g2comparison} we show that for $g=2$ the latter inclusion is strict also when $n=4,5$. This, combined with Theorem~\ref{theoremintro}, completely answers the related question.
}
    \subsection{Acknowledgements}
	I thank my supervisor Nicola Pagani for the many useful insights and conversations. I would like to acknowledge Rhys Wells for his contribution on the computational side. I am also grateful to Filippo Viviani for all the helpful comments and suggestions on the previous version of the paper.

	\section{Contractible stability conditions for graphs}

	\subsection{Graphs}
	We recall the notation concerning graphs from \cite{MMUV} and \cite{viviani2023new}, which we will adopt throughout the paper.
	
	Let $\Gamma$ be a connected vertex-weighted graph of vertices $V(\Gamma)$ and edges $E(\Gamma)$. Its \emph{total genus} $g(\Gamma)$ is defined as
	\begin{equation*}
		g(\Gamma)=V(\Gamma)-E(\Gamma)+1+\sum_{v\in V(\Gamma)}g(v),
	\end{equation*}
	where $g(v)$ denotes the genus (i.e. weight) of the vertex $v$.
	
	Given a set of vertices $W\subset V(\Gamma)$, we define the \emph{induced subgraph $\Gamma[W]$} to be the graph whose vertex set is $W$ and edge set is given by the edges of $\Gamma$ that connect only vertices in $W$. 
	
	Given a graph $\Gamma$ and a subgraph $\Gamma'\subset\Gamma$, we define its complementary $\Gamma'^{\mathsf{c}}$ to be the induced subgraph $\Gamma[V(\Gamma)\setminus V(\Gamma')]$.
	We will be mainly interested in connected subgraphs whose complementary subgraph is connected too. Such subgraphs are said to be \emph{biconnected}.
	
	Let $\Gamma'\subset\Gamma$ be a marked weighted connected subgraph. We denote the genus of $\Gamma'$ by $g(\Gamma')$ and its set of marked flags by $A(\Gamma')$. Moreover, with a slight abuse of notation, if $\Gamma'=\Gamma[W]$ for some $W\subset V(\Gamma)$, we define $g(\Gamma')=g(W)$ and $A(\Gamma')=A(W)$.
	
	Let $\Gamma',\Gamma''\subsetneq\Gamma$ be disjoint subgraphs. We define the \emph{valence} of $(\Gamma',\Gamma'')$ to be the number
	\begin{equation*}
		\text{val}(\Gamma',\Gamma'')=|E(\Gamma',\Gamma'')|,
	\end{equation*}	
	where $E(\Gamma',\Gamma'')$ denotes the subset of $E(\Gamma)$ whose elements connect a vertex of $\Gamma'$ with one of $\Gamma''$.
	Analogously, we define the \emph{valence} of $\Gamma'$ to be 
	\begin{equation*}
		\text{val}(\Gamma')=|E(\Gamma',\Gamma'^{\mathsf{c}})|.
	\end{equation*}
	As above, we use $\text{val}(W,Z)$ (resp. $\text{val}(W)$) to signify $\text{val}(\Gamma[W],\Gamma[Z])$ (resp. $\text{val}(\Gamma[W])$), in order to lighten the writing. 
	Notice that for a marked graph the notion of valence doesn't take into account marked flags.

	We can now pass to the notions of stable graphs and their category, which will be crucial for the rest of the paper.
	
	\begin{definition}
		 A marked weighted graph is \emph{stable} if, for each vertex $v$,
		 \begin{equation*}
		 	2g(v)-2+\text{val}(v)+n(v)>0,
		 \end{equation*}
		 where $n(v)=|A(v)|$ denotes the cardinality of the set of marked flags emanating from $v$.
		 
		 Let $(g,n)$ be an hyperbolic pair, that is $g,n\in\ZZ_{\geq 0}$ such that $2g-2+n>0$. We define (a skeleton of) the category $\Ggn$, whose objects are a choice of one stable $n$-pointed graph of genus $g$  for every isomorphism classes, and morphisms $c:\Gamma\to\Gamma'$ are contractions of edges of $\Gamma$ followed by an isomorphism.		 
	\end{definition} 

    A morphism of graphs $c:\Gamma\to\Gamma'$ induces a surjective map $c:V(\Gamma)\to V(\Gamma')$ between their sets of vertices.
    
	Particular relevance will be held by \emph{vine graphs}, i.e. graphs with two vertices and no loops (a loop is an edge with its two endpoints coinciding). We denote a vine graph in $\Ggn$ by $V(\text{val}(v_1),g(v_1),A(v_1))$, where $v_1$ is one of the two vertices. Clearly, the two vine graphs $V(e,h,A)$ and $V(e,g+1-e-h,[n]\setminus A)$ are isomorphic, and, when $n=0$, swapping the two vertices of a vine graph $V(g-2h+1,h,\emptyset)$ induces an automorphism.

	\subsection{Universal and c-stability conditions}
	In this section we introduce the notion of universal stability condition on $\Ggn$ and restrict it to a $c$-stability condition on a graph $\Gamma$. We show that the notion is equivalent to the one of $V$-stability condition on the graph $\Gamma$ that is compatible with morphisms of $\Ggn$. Moreover, a classical numerical polarisation compatible with $\Ggn$-morphisms on such a graph always induces a $c$-stability condition, while viceversa is not necessarily true.
	\begin{definition}
		Let $(g,n)\in \ZZ_{\geq 0}^2$ be a hyperbolic pair. The \emph{stability domain of $\Ggn$} is 
		\begin{equation*}
			\D_{g,n}=\{(e,h,A):1\leq e\leq g+1,0\leq h\leq g+1-e, A\subseteq [n]\}\setminus \{(2,0,\emptyset),(2,g-1,[n])\}.
		\end{equation*} 
	\end{definition}
	
	We observe that the triples in $\D_{g,n}$ correspond bijectively to subgraphs (i.e. vertices) of graphs of vine type in $\Ggn$, up to isomorphisms. 
    Explicitly:
	\begin{proposition}
		The map 
		\begin{align*}
			\{(\Gamma,v):\text{$\Gamma$ is a stable vine graph and }v\in V(\Gamma)\}_{/\sim}&\to \D_{g,n}\\
			[(\Gamma,v)]&\mapsto(\textup{val}(v),g(v),A(v))
		\end{align*}
		is a bijection with inverse $(e,h,A)\mapsto [(V(e,h,A),v_1)]$. Here two pairs $(\Gamma,v)$ and $(\Gamma',v')$ are equivalent if there is an isomorphism $\phi:\Gamma\to\Gamma'$ such that $\phi(v)=v'$.
	\end{proposition}
	
	\begin{definition}\label{cstabuniv}
		A \emph{ degree $d$ universal stability condition of type $(g,n)$} is a collection of integers 
		\begin{equation*}
			\m=\{\mEhA:(e,h,A)\in\D_{g,n}\},
		\end{equation*}
		which satisfies the following properties:
		\begin{itemize}
			\item[i.] For each $(e,h,A)\in \D_{g,n}$,
			\begin{equation}\label{prop1}
				\mEhA+\m_{(e,g+1-e-h,A^\mathsf{c})}=d+1-e;
			\end{equation}
			\item[ii.] For each $(e,h,A),(e',h',A'),(e'',h'',A'')\in\D_{g,n}: A=A'\cup A'', A'\cap A''=\emptyset, \text{ } 2(h+1-(h'+h''))=e'+e''-e$, and there exists a triangle whose edges have length $e,e'$ and $e''$ respectively,
			\begin{equation}\label{prop2}
				0\leq \mEhA-h-(\m_{(e',h',A')}-h'+\m_{(e'',h'',A'')}-h'')\leq 1.
			\end{equation}
		\end{itemize}
	\end{definition}
	To make sense of the definition, we observe that the required numerical conditions are related to graphs in $\Ggn$ via the following lemmas:
	\begin{lemma}
		Let $(e',h',A')$ and $(e'',h'',A'')\in \D_{g,n}$. There exists a graph $\Gamma\in\Ggn$ with biconnected and nontrivial subgraphs $\Gamma'$ and $\Gamma''$ with $(\textup{val}(\Gamma'),g(\Gamma'),A(\Gamma'))=(e',h',A')$ and $(\textup{val}(\Gamma''),g(\Gamma''),A(\Gamma''))=(e'',h'',A'')$, such that $\Gamma=\overline{\Gamma'\cup\Gamma''}$ and $\Gamma'\cap\Gamma''=\emptyset$ if and only if $(e'',h'',A'')=(e',g+1-e'-h',A'^{\mathsf{c}})$.
	\end{lemma}
	\begin{proof}
		The "only if" direction is trivial by definition of a biconnected subgraph, while the "if" direction is proven by the existence of the vine graph $V(e',h',A')\in\Ggn$.
	\end{proof}
	\begin{lemma}\label{trianglelemma}
		Let $(e,h,A),(e',h',A')$ and $(e'',h'',A'')\in \D_{g,n}$. There exists a graph $\Gamma\in\Ggn$ with biconnected and nontrivial subgraphs $\Gamma^0,\Gamma'$ and $\Gamma''$ with, respectively,  $(\textup{val}(\Gamma^0),g(\Gamma^0),A(\Gamma^0))=(e,h,A)$, $(\textup{val}(\Gamma'),g(\Gamma'),A(\Gamma'))=(e',h',A')$ and $(\textup{val}(\Gamma''),g(\Gamma''),A(\Gamma''))=(e'',h'',A'')$, such that $\Gamma^0=\overline{\Gamma'\cup\Gamma''}$ and $\Gamma'\cap\Gamma''=\emptyset$ if and only if $A=A'\cup A'', A'\cap A''=\emptyset, \text{ } 2(h+1-(h'+h''))=e'+e''-e$, and there exists a triangle whose edges have length $e,e'$ and $e''$ respectively.
	\end{lemma}
	\begin{proof}
		\begin{itemize}
			\item [$\Rightarrow$] Let $\Gamma\in\Ggn$ with subgraphs $\Gamma^0,\Gamma'$ and $\Gamma''$ as in the statement. Then, $\text{val}(\Gamma',\Gamma'')=h+1-(h'+h'')$, $e'=\text{val}(\Gamma',\Gamma'')+\text{val}(\Gamma',{\Gamma^0}^\mathsf{c})$ and $e''=\text{val}(\Gamma',\Gamma'')+\text{val}(\Gamma'',{\Gamma^0}^\mathsf{c})$, with both $\text{val}(\Gamma',{\Gamma^0}^\mathsf{c})$ and $\text{val}(\Gamma'',{\Gamma^0}^\mathsf{c})$ different from $0$ and from $\text{val}(\Gamma^0,{\Gamma^0}^\mathsf{c})$, since $\Gamma'$ and $\Gamma''$ are biconnected and nontrivial. As $e=\text{val}(\Gamma^0,{\Gamma^0}^\mathsf{c})=\text{val}(\Gamma',{\Gamma^0}^\mathsf{c})+\text{val}(\Gamma'',{\Gamma^0}^\mathsf{c})$, we have
                \begin{equation*}
                    2(h+1-(h'+h''))=e'+e''-e.
                \end{equation*} 
			Moreover, 
			\begin{equation*}
				e=\text{val}(\Gamma',{\Gamma^0}^\mathsf{c})+\text{val}(\Gamma'',{\Gamma^0}^\mathsf{c})<\text{val}(\Gamma',{\Gamma^0}^\mathsf{c})+\text{val}(\Gamma'',{\Gamma^0}^\mathsf{c})+2\text{val}(\Gamma',\Gamma'')=e'+e'',
			\end{equation*}
			iff $\text{val}(\Gamma',\Gamma'')\neq 0$, that is, if and only if $\Gamma^0$ is connected, and
			\begin{equation*}
				e'=\text{val}(\Gamma',\Gamma'')+\text{val}(\Gamma',{\Gamma^0}^\mathsf{c})<\text{val}(\Gamma',\Gamma'')+\text{val}(\Gamma'',{\Gamma^0}^\mathsf{c})+\text{val}(\Gamma^0,{\Gamma^0}^\mathsf{c})=e''+e
			\end{equation*}
			iff $\text{val}(\Gamma'',{\Gamma^0}^\mathsf{c})> 0$, that is, if and only if $\Gamma'$ is nontrivial and its complementary is connected. The analogous statement is true for $e''$. Therefore, since the conditions on the sets $A,A',A''$ are trivially satisfied, we conclude that this implication holds.
			\item[$\Leftarrow$] Define $k_{12}:=\frac{e'+e''-e}{2}, k_{01}:=\frac{e+e'-e''}{2}$ and $k_{02}:=\frac{e+e''-e'}{2}$. Since $2(h+1-(h'+h''))=e'+e''-e$, and $h'+h''\leq h$, we deduce that $k_{12} \in \ZZ_{\geq1}$. Hence, by the assumption of the existence of a triangle with edges of length $e,e',e''$, $k_{01},k_{02}\in\ZZ_{\geq1}$ too, and $e=k_{01}+k_{02},e'=k_{01}+k_{12}$ and $e''=k_{02}+k_{12}$. Therefore, the graph of Figure \ref{triangle diagram}, where $h^\mathsf{c}=g+1-e-h$, and the numbers $k_{ij}$ denote the number of edges, proves our claim.
			
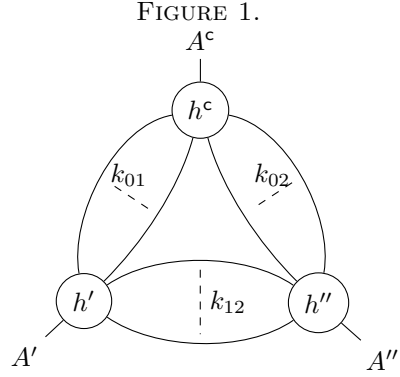
\begin{figure}[ht!]
				\caption{}
				\label{triangle diagram}
				\centering
				\begin{tikzpicture}
					\node (h0c) [shape=circle,draw] {$h^\mathsf{c}$};
					\node (Ac)[above=.3cm of h0c] {$A^\mathsf{c}$};
					\node [shape=circle,draw,below left=2cm and 1cm of h0c] (h1){$h'$};
					\node (A1)[below left=.3cm of h1] {$A'$};
					\node [shape=circle,draw,below right=2cm and 1cm of h0c] (h2){$h''$};
					\node (A2)[below right=.3cm of h2] {$A''$};

					\path[-,draw]
					(h1) edge[looseness=0.8, out=35, in=145] node[name=l11] {}(h2)
					(h1) edge[looseness=0.8, out=-35, in=-145] node[name=l12] {} (h2);
					\path[-,draw]
					(h1) edge[looseness=0.8, out=45, in=-105] node[name=l01] {}(h0c)
					(h1) edge[looseness=0.8, out=100, in=190] node[name=l02] {} (h0c);
					\path[-,draw]
					(h2) edge[looseness=0.8, out=80, in=-10] node[name=l21] {}(h0c)
					(h2) edge[looseness=0.8, out=135, in=-75] node[name=l22] {} (h0c);
					
					\draw (l11) to node[right]{$k_{12}$}(l12)[dashed] ;
					\draw (l21) to node[above]{$k_{02}$}(l22)[dashed];
					\draw (l01) to node[above]{$k_{01}$}(l02)[dashed];
					
					\draw (h0c) to (Ac);
					\draw (h1) to (A1);
					\draw (h2) to (A2);

				\end{tikzpicture}
			\end{figure}
		\end{itemize}
		
	\end{proof}

	\begin{definition}
		Let $\Gamma\in\Ggn$ be a graph. We define the \emph{stability domain of $\Gamma$} as 
		\begin{equation*}
			\D(\Gamma)=\{(e,h,A)\in\D_{g,n}:\exists\textup{ }\emptyset\subsetneq\Gamma'\subsetneq\Gamma, c:\Gamma\to V(e,h,A) \textup{ and }c^{-1}(v_1)=\Gamma'\}.
		\end{equation*}
	\end{definition}
	
	Clearly, such a morphism $c:\Gamma\to V(e,h,A)$ exists if and only if the subgraph $\Gamma'$ is biconnected, nontrivial and such that $(\text{val}(\Gamma'),g(\Gamma'),A(\Gamma'))=(e,h,A)$
	\begin{definition}\label{cstabGamma}
		Let $\Gamma\in\Ggn$ be a graph. A \emph{degree $d$ cone-compatible stability condition on $\Gamma$} (\emph{$c$-stability condition} for short) is a collection of integers 
		\begin{equation*}
			\m^\Gamma=\{\mEhA^\Gamma:(e,h,A)\in\D(\Gamma)\}
		\end{equation*}
		satisfying the following properties:
		\begin{itemize}
			\item [i.] $\forall (e,h,A)\in\D(\Gamma)$,
			\begin{equation}\label{prop1Gamma}
				\mEhA^\Gamma+\m^\Gamma_{(e,g+1-e-h,A^\mathsf{c})}=d+1-e;
			\end{equation}
			\item[ii.] $\forall \text{ } (e,h,A),(e',h',A'),(e'',h'',A'')\in\D(\Gamma): A=A'\cup A'', A'\cap A''=\emptyset, \text{ } 2(h+1-(h'+h''))=e'+e''-e$, and there exists a triangle whose edges have length $e,e'$ and $e''$ respectively,
			\begin{equation}\label{prop2Gamma}
				0\leq \mEhA^\Gamma-h-(\m^\Gamma_{(e',h',A')}-h'+\m^\Gamma_{(e'',h'',A'')}-h'')\leq 1.
			\end{equation}
		\end{itemize}		
	\end{definition}
	
	Let $\{\m^\Gamma\}_{\Gamma\in\Ggn}$ be an assignment of $c$-stability conditions that agree on common domains, i.e. such that, for each $\Gamma,\Gamma'\in\Ggn$ and $(e,h,A)\in\D(\Gamma)\cap\D(\Gamma')$, $\mEhA^\Gamma=\mEhA^{\Gamma'}$. Taking the limit over $\Ggn$ gives a universal stability condition $\m$, which coincides with $\m^\Gamma$ on each $\Gamma\in\Ggn$. That is, the following holds
	\begin{lemma}\label{gluelemma}
		
		The map 
		\begin{align*}
			\{\textup{Universal }\textup{stability conditions of type $(g,n)$}\}&\to\begin{Bmatrix}
				\{c\textup{-stability conditions on $\Gamma$}\}_{\Gamma\in\Ggn}\\ \textup{that agree on common domains}
			\end{Bmatrix}\\
			\m&\mapsto\{\m_{|_{\D(\Gamma)}}\}_{\Gamma\in\Ggn}
		\end{align*}
		is a bijection with inverse 
		\begin{align*}
			\begin{Bmatrix}
				\{c\textup{-stability conditions on $\Gamma$}\}_{\Gamma\in\Ggn}\\ \textup{that agree on common domains}
			\end{Bmatrix}&\to\{\textup{Universal }\textup{stability conditions of type $(g,n)$}\}\\
			\{\m^\Gamma\}_{\Gamma\in\Ggn}&\mapsto \varprojlim_{\Ggn}\m^\Gamma,
		\end{align*}
		where $(\varprojlim_{\Ggn}\m^\Gamma)_{(e,h,A)}=\mEhA^\Gamma$, for any graph $\Gamma$ such that $(e,h,A)\in\D(\Gamma)$.
	\end{lemma}
	\begin{proof}
		Let $\{\m^\Gamma\}_{\Gamma\in\Ggn}$ be in the domain of the latter map.
		Let $(e,h,A)\in\D_{g,n}$, and let $\m^\Gamma$ be a $c$-stability condition such that $(e,h,A)\in\D(\Gamma)$. For each other $\Gamma'$ such that $(e,h,A)\in\D(\Gamma')$, one must have $\mEhA^\Gamma=\mEhA^{\Gamma'}$, therefore $(\varprojlim_{\Ggn}\m^{\Gamma'})_{(e,h,A)}:=\mEhA^\Gamma$ is well defined. Everything else is as trivial.		 
	\end{proof}
	
	We, now, want to show what is the relation between $V$-stability conditions, as defined in \cite[Definition~1.4]{viviani2023new}, and $c$-stability conditions on a given graph $\Gamma\in\Ggn$.
	
	\begin{lemma}
		Let $\C(\Gamma)=\{Z\subset V(\Gamma):\emptyset\subsetneq\Gamma(Z)\subsetneq\Gamma \text{ is biconnected}\}$. Then the map
		\begin{align*}
			\alpha:\C(\Gamma)&\to \D(\Gamma)\\
			Z&\mapsto (\textup{val}(Z),g(Z),A(Z))
		\end{align*}
		is surjective.
	\end{lemma}
	\begin{proof}
		By definition of $\D(\Gamma)$, for each $(e,h,A)\in\D(\Gamma)$, there exists a nontrivial and biconnected subgraph $\Gamma'\subsetneq\Gamma$ such that $(\text{val}(\Gamma'),g(\Gamma'),A(\Gamma'))=(e,h,A)$. Thus,  $Z=V(\Gamma')\in \alpha^{-1}(e,h,A)$.
	\end{proof}
	
	We recall that morphisms in $\Ggn$ are contractions of prescribed edges followed by an isomorphism, thus they induce surjections on the level of subsets of vertices. In particular, if $c:\Gamma\to\Gamma'$ is a morphism, for each $Z\subset V(\Gamma')$, there exists a maximal $W\subset V(\Gamma)$ such that $c(W)=Z$. This fact ensures the following definition is well posed. 
	
	\begin{definition}\label{compatiblewithmorphs}
		Let $\Gamma\in\Ggn$ be a graph and $\n^\Gamma$ an assignment of integers 
		\begin{equation*}
			\n^\Gamma=\{\n_Z^\Gamma:Z\in\C(\Gamma)\}.
		\end{equation*}
		We say that $\n^\Gamma$ is \emph{compatible with $\Ggn$-morphisms} if for each $\Gamma'\in\Ggn$ and every two morphisms $c:\Gamma\to\Gamma'$ and $\tilde{c}:\Gamma\to\Gamma'$ in $\Ggn$, it satisfies
		\begin{equation*}
			\n^\Gamma_{c^{-1}(Z)}=\n^\Gamma_{\tilde{c}^{-1}(Z)},
		\end{equation*}
		for each $Z\in\C(\Gamma')$.
	\end{definition}

	With this at hand, we can finally show that $c$-stability conditions are precisely $V$-stability conditions that satisfy compatibility with $\Ggn$-morphisms.
	
	\begin{proposition}\label{cstab=vstab}
		Let $\Gamma\in\Ggn$ be a graph. The map
		\begin{align}
			\beta:\{\textup{$c$-stabilities on $\Gamma$}\}&\to \{\textup{$V$-stabilities on $\Gamma$ compatible with $\Ggn$-morphisms}\}\\
			\m^\Gamma&\mapsto\n^\Gamma:=\{\nZ^\Gamma:=\m^\Gamma_{(\textup{val}(Z),g(Z),A(Z))}:Z\in\C(\Gamma)\}\notag
		\end{align}
		is a bijection with inverse
		\begin{equation}
			\gamma:\n^\Gamma\mapsto\m^\Gamma:=\{\mEhA^\Gamma:=\n^\Gamma_{Z(e,h,A)}:(e,h,A)\in\D(\Gamma)\},
		\end{equation}
		where $Z(e,h,A)\in \alpha^{-1}(e,h,A)$.
	\end{proposition}
	\begin{proof}
		First, we show that for each degree $d$ $c$-stability condition $\m^\Gamma$, $\beta(\m^\Gamma)$ is indeed a $V$-stability condition of degree $d$ on $\Gamma$ compatible with $\Ggn$-morphisms. As the compatibility requirement is trivially satisfied by definition, we only need to prove that $\beta(\m^\Gamma)=\{\m^\Gamma_{(\textup{val}(Z),g(Z),A(Z))}:Z\in\C(\Gamma)\}$ satisfies Properties~(1) and~(2) of \cite[Definition~1.4]{viviani2023new}.
		\begin{itemize}
			\item[(1)] For each $Z\in \C(\Gamma)$, there exists a contraction $c:\Gamma\to V(\textup{val}(Z),g(Z),A(Z))$ such that $\Gamma(Z)=c^{-1}(v_1)$ and $\Gamma(Z^\mathsf{c})=c^{-1}(v_2)$, where $v_2$ denotes the remaining vertex of the vine graph. Therefore the triples $(\textup{val}(Z),g(Z),A(Z))$ and $(\textup{val}(Z^\mathsf{c}),h(Z^\mathsf{c}),A(Z^\mathsf{c}))$ are elements of $\D(\Gamma)$ that satisfy 
			\begin{equation*}
				\textup{val}(Z^\mathsf{c})=\textup{val}(Z),\qquad g(Z^\mathsf{c})=g+1-e-g(Z),\qquad A(Z^\mathsf{c})=A(Z)^\mathsf{c}.
			\end{equation*}
			Thus property $(1)$ is satisfied, since
			\begin{equation*}
				\beta(\m^\Gamma)_Z+\beta(\m^\Gamma)_{Z^\mathsf{c}}=\m^\Gamma_{(\textup{val}(Z),g(Z),A(Z))}+\m^\Gamma_{(\textup{val}(Z),g+1-e-g(Z),A(Z)^\mathsf{c})}=d+1-\textup{val}(Z),
			\end{equation*}
			where the last equality is given by \eqref{prop1Gamma}.
			\item[(2)] Let $Z,Z',Z''\in\C(\Gamma)$ such that $Z'\cap Z''=\emptyset$ and $Z'\cup Z''=Z$. By the biconnectedness hypotesis, $\alpha(Z),\alpha(Z'),\alpha(Z'')\in\D(\Gamma)$ and satisfy 
			\begin{align*}
				\textup{val}(Z)&=\textup{val}(Z')+\textup{val}(Z'')-2\text{val}(Z',Z''),\\	g(Z)&=g(Z')+g(Z'')+\text{val}(Z',Z'')-1,\\
				A(Z)&=A(Z')\cup A(Z'') \text{ and }A(Z')\cap A(Z'')=\emptyset.
			\end{align*}
			Hence $2(g(Z)+1-(g(Z')+g(Z'')))=\textup{val}(Z')+\textup{val}(Z'')-\textup{val}(Z)$, there exists a triangle with edge lengths $e,e',e''$, by the proof of Lemma \ref{trianglelemma}, and, by \eqref{prop2Gamma}, 
			\begin{multline*}
				\beta(\m^\Gamma)_Z-\beta(\m^\Gamma)_{Z'}-\beta(\m^\Gamma)_{Z''}-\text{val}(Z',Z'')=\\=\m^\Gamma_{(\textup{val}(Z),g(Z),A(Z))}-\m^\Gamma_{(\textup{val}(Z'),g(Z'),A(Z'))}-\m^\Gamma_{(\textup{val}(Z''),g(Z''),A(Z''))}-g(Z)+g(Z')+g(Z'')-1\in \{-1,0\}.
			\end{multline*}
		\end{itemize}
		Therefore the map $\beta$ is well defined.
		
		To show the same for $\gamma$, we, first, observe that if $Z,W\in \C^d(\Gamma)$ are such that $\alpha(Z)=\alpha(W)$ and $c_Z:\Gamma\to V(\textup{val}(Z),g(Z),A(Z)),$ $c_W:\Gamma\to V(\textup{val}(W),g(W),A(W))$ are the respective contractions, then $c_Z(Z)= c_W(W)$, and, by compatibility with $\Ggn$-morphisms $\nZ^\Gamma=\n^\Gamma_{c_Z^{-1}(c_Z(Z))}=\n^\Gamma_{c_W^{-1}(c_W(W))}=\n^\Gamma_W$. Therefore, $\gamma$ is a well defined function to $\ZZ^{|\D(\Gamma)|}$.
		Now we check that for each degree $d$ $V$-stability condition compatible with $\Ggn$-morphisms $\n^\Gamma$, 
        $\gamma(\n^\Gamma)$ is a $c$-stability condition of the same degree, i.e. it satisfies the requirements of Definition~\ref{cstabGamma}.
		\begin{itemize}
			\item[i.] Let $(e,h,A)\in \D(\Gamma)$ and $Z\in \alpha^{-1}(e,h,A)$. We notice that $Z^\mathsf{c}\in \alpha^{-1}(e,g+1-h-e,A^\mathsf{c})$. Thus property i. is satisfied, since
			\begin{equation*}
				\gamma(\n^\Gamma)_{(e,h,A)}+\gamma(\n^\Gamma)_{(e,g+1-h-e,A^\mathsf{c})}=\nZ^\Gamma+\n^\Gamma_{Z^\mathsf{c}}=d+1-\text{val}(Z,Z^\mathsf{c})=d+1-e.
			\end{equation*}
			\item[ii.] Let $(e,h,A),(e',h',A'),(e'',h'',A'')\in\D(\Gamma)$ such that $ A=A'\cup A'', A'\cap A''=\emptyset$, $ 2(h+1-(h'+h''))=e'+e''-e$ and there exists a triangle whose edges have length $e,e'$ and $e''$ respectively. By definition of $\D(\Gamma)$, there exist $\Gamma^0,\Gamma',\Gamma''\subsetneq\Gamma$ such that $(\textup{val}(\Gamma^0),h(\Gamma^0),A(\Gamma^0))=~(e,h,A)$, $(\textup{val}(\Gamma'),h(\Gamma'),A(\Gamma'))=(e',h',A')$ and $(\textup{val}(\Gamma''),h(\Gamma''),A(\Gamma''))=(e'',h'',A'')$, and, by choosing them as in Lemma~\ref{trianglelemma}, we can assume $\Gamma'\cap\Gamma''=\emptyset$ and $\overline{\Gamma'\cup\Gamma''}=\Gamma^0$.
			Then, $Z=V(\Gamma^0),Z'=V(\Gamma')$ and $Z''=V(\Gamma'')$, satisfy $Z\in\alpha^{-1}(e,h,A),$ $Z'\in\alpha^{-1}(e',h',A')$ and $Z''\in\alpha^{-1}(e'',h'',A'')$. Thus property ii. holds as
			\begin{multline*}
				\gamma(\n^\Gamma)_{(e,h,A)}-\gamma(\n^\Gamma)_{(e',h',A')}-\gamma(\n^\Gamma)_{(e'',h'',A'')}-h+h'+h''=\\=\n^\Gamma_Z-\n^\Gamma_{Z'}-\n^\Gamma_{Z''}-\text{val}(Z',Z'')+1\in \{0,1\}.
			\end{multline*}
		\end{itemize}
		
		Finally, since both $\beta$ and $\gamma$ are well defined, we just need to show that they are indeed mutual inverses:
		\begin{equation*}
			\n^\Gamma\xmapsto{\gamma}\{\mEhA^\Gamma:=\n^\Gamma_{Z(e,h,A)}\}\xmapsto{\beta}\{\tilde{\n}^\Gamma_Z:=\m^\Gamma_{(\textup{val}(Z),g(Z),A(Z))}\}=\n^\Gamma,
		\end{equation*}
		since, for each $W\in\alpha^{-1}(\textup{val}(Z),g(Z),A(Z))$, $\n^\Gamma_W=\nZ^\Gamma$.\\
		Viceversa, 
		\begin{equation*}
			\m^\Gamma\xmapsto{\beta}\{\nZ^\Gamma:=\m^\Gamma_{(\textup{val}(Z),g(Z),A(Z))}\}\xmapsto{\gamma}\{\tilde{\m}^\Gamma_{(e,h,A)}:=\n^\Gamma_{Z(e,h,A)}\}=\m^\Gamma,
		\end{equation*}
		since $(\textup{val}(Z(e,h,A)),g(Z(e,h,A)),A(Z(e,h,A)))=(e,h,A)$, for each $(e,h,A)\in\D(\Gamma)$.
	\end{proof}
	
	Therefore we notice that $c$-stability conditions are just $V$-stability conditions that satisfy the extra requirement of being compatible with $\Ggn$-morphisms. This means that they are compatible with respect to all the morphisms $c:\Gamma\to\Gamma'$, whose images are graphs in the cone over $\Gamma$ in $\Ggn$, hence the name of "cone-compatible" stability conditions.
	
	\begin{remark}
		Notice that, in general, given a graph $\Gamma\in\Ggn$, not all the $V$-stability conditions on $\Gamma$ are necessarily compatible with $\Ggn$-morphisms, as shown by \cite[Example~1.27]{viviani2023new}. Indeed, let $\Gamma$ be the graph in loc. cit. decorated with markings $\{1\}$ and $\{2\}$ at the vertices $5$ and $6$ respectively, and let $\n^\Gamma$ be the $V$-stability condition associated to the mildly superadditive function $\beta$. Notice that $\Gamma$ is isomorphic to a graph in $G_{3,2}$. Consider the morphisms $c:\Gamma\to V(3,0,\{2\})$, that contracts edges $\{\overline{16},\overline{13},\overline{24},\overline{45}\}$ and $\tilde{c}:\Gamma\to V(3,0,\{2\})$, that contracts edges $\{\overline{16},\overline{12},\overline{35},\overline{45}\}$, in the notation of loc. cit. Then, we observe that 
		\begin{equation*}
			0=\n^\Gamma_{c^{-1}(\{v_1\})}\neq\n^\Gamma_{\tilde{c}^{-1}(\{v_1\})}=1.
		\end{equation*}
	\end{remark}
	
	As previously announced, we now show how $\Ggn$-morphisms compatible numerical polarisations induce $c$-stabilities, on a given graph $\Gamma$ as well as in the universal case. 
	Recalling the language of \cite{Oda1979CompactificationsOT}, we give the following:
	
	\begin{definition}\label{fclassicalcj}
		A \emph{degree $d$ numerical polarisation} $\phi^\Gamma$ on a graph $\Gamma\in\Ggn$ is a real divisor $\phi^\Gamma\in\textup{Div}(\Gamma)\otimes\RR=\RR^{V(\Gamma)}$ of total degree $|\phi^\Gamma |=\sum_{v\in V(\Gamma)}\phi^\Gamma_v=d$. 
	
		Such a polarisation is said to be \emph{general} if it further satisfies the property
		\begin{equation*}
			\phi^\Gamma_W -\frac{\textup{val}(W)}{2}\notin\ZZ,
		\end{equation*}
		for each nontrivial biconnected $W\subset V(\Gamma)$. Here $\phi^\Gamma_W=\sum_{v\in W}{\phi^\Gamma_v}$.
		
		We say that $\phi^\Gamma$ is \emph{compatible with $\Ggn$-morphisms} if for each $\Gamma'\in\Ggn$ and each two morphisms $c:\Gamma\to\Gamma'$ and $\tilde{c}:\Gamma\to\Gamma'$ in $\Ggn$, it satisfies
		\begin{equation*}
			\phi^\Gamma_{c^{-1}(Z)}=\phi^\Gamma_{\tilde{c}^{-1}(Z)},
		\end{equation*}
		for each $Z\in\C(\Gamma')$.
	\end{definition}
	Similarly, in the universal case we give the following definition, originally stated in \cite[Definition~3.1]{Kass_2019}.
	\begin{definition}\label{phiunivdef}
       Let $\{\phi^\Gamma\}_{\Gamma\in\Ggn}$ be a family of numerical polarisations which are \emph{compatible with contractions and automorphisms} in the sense of \cite[Section $3$]{Kass_2019}.
       	A \emph{degree $d$ universal numerical polarisation} is the limit 

        \begin{equation*}
            \phi:=\varprojlim_{\Ggn}\phi^\Gamma.
        \end{equation*}
  
	\end{definition}

        \begin{proposition-definition}\label{phitocdef}
        Let $\Gamma\in\Ggn$ be a graph, and let $\phi^\Gamma$ be a degree $d$ general numerical polarisation on $\Gamma$, compatible with $\Ggn$-morphisms. We define the \emph{classical $c$-stability condition $\m^\Gamma(\phi^\Gamma)$ induced by $\phi^\Gamma$ on $\Gamma$} by the assignment
        \begin{equation*}
            \m^\Gamma(\phi^\Gamma):=\{\mEhA^\Gamma(\phi^\Gamma):=\Bigl\lceil\phi^\Gamma_{Z(e,h,A)}-\frac{e}{2}\Bigr\rceil:(e,h,A)\in\D(\Gamma)\},
        \end{equation*}
        where $Z(e,h,A)\in\alpha^{-1}(e,h,A)$.
        \end{proposition-definition}
        
        \begin{proof}
        The induced stability condition $\m^\Gamma(\phi^\Gamma)$ is the image, via the map of Proposition~\ref{cstab=vstab}, of the $V$-stability condition 
        \begin{equation*}
			\n^\Gamma(\phi^\Gamma):=\{\n^\Gamma(\phi^\Gamma)_Z=\Bigl\lceil\phi^\Gamma_Z-\frac{\textup{val}(Z)}{2}\Bigr\rceil:Z\in\C(\Gamma)\},
	\end{equation*}
	induced by $\phi^\Gamma$, defined in \cite[Example~1.5]{viviani2023new}.
        
        Thus, to show that $\m^\Gamma(\phi^\Gamma)$ is well defined, it is sufficient to notice that whenever $\phi^\Gamma$ is compatible with $\Ggn$-morphisms, so is the induced $V$-stability condition.
	Indeed, this is true by the following     observation. Let $\Gamma'\in\Ggn$ and $c,\tilde{c}:\Gamma\to\Gamma'$ be morphisms. For each $Z\in V(\Gamma')$, 
			\begin{equation*}
				\textup{val}(c^{-1}(Z))=\textup{val}(Z)=\textup{val}(\tilde{c}^{-1}(Z)),
			\end{equation*}
			since both $c$ and $\tilde{c}$ only contract edges joining vertices in $c^{-1}(Z)$ and $\tilde{c}^{-1}(Z)$ respectively.
        \end{proof}

        A more explicit form for the classical stability condition $\m^\Gamma(\phi^\Gamma)$ will prove useful for constructing examples in the following section, thus we show:
 
	\begin{proposition}\label{phitocprop}
		Let $g\geq 2$, and let $\Gamma\in\Ggn$ be a $2$-connected graph (i.e. a graph without separating edges). Let $\phi^\Gamma$ be a general degree $d$ numerical polarisation on $\Gamma$, compatible with $\Ggn$-morphisms. Then, the $\phi^\Gamma$-induced classical $c$-stability condition $\m^\Gamma(\phi^\Gamma)$ satisfies
	\begin{equation*}
			\mEhA^\Gamma(\phi^\Gamma)=	\Bigl\lceil\frac{(e+2h-2)(d+1-g)+(2g-e-2h)\phi^\Gamma_A-(e+2h-2)\phi^\Gamma_{A^\mathsf{c}}}{2g-2}  \Bigr\rceil +h-1,
		\end{equation*}
            where $\phi^\Gamma_A=\sum_{v:A(v)\subset A}\phi^\Gamma_v$, for $v\in V(\Gamma)$.
	\end{proposition}
	\begin{proof}		
		We observe that the $V$-stability condition $\n(\phi^\Gamma)$ induced by a numerical polarisation $\phi^\Gamma$ of \cite[Example~1.5]{viviani2023new} depends on the stability polytope containing $\phi^\Gamma$, rather than on the specific $\phi^\Gamma$.
		
		Let $c:\Gamma'\to\Gamma$, be a morphism, and let $\phi^\Gamma$ be a numerical polarisation on $\Gamma$. By generality of $\phi^\Gamma$, there exists a numerical polarisation $\phi^{\Gamma'}$ on $\Gamma'$ such that the pushforward $c(\phi^{\Gamma'})$ is in the same polytope as $\phi^\Gamma$. Therefore it is enough to prove our statement on the initial objects of $\Ggn$, that is, trivalent graphs.	
        
       So, we assume that $\Gamma$ is trivalent. By compatibility with $\Ggn$-morphisms, since $\Gamma$ is $2$-connected, we have $\phi^\Gamma_{v_1}=\phi^\Gamma_{v_2}$ for any pair of unlabeled genus $0$ vertices $v_1,v_2$. Hence, the classical $c$-stability condition induced by $\phi^\Gamma$ on a $2$-connected trivalent graph in $\Ggn$ has elements 
		\begin{equation*}
			\mEhA^\Gamma(\phi^\Gamma)=	\Bigl\lceil\frac{k(d+1-g)+(2g-2-k)\phi^\Gamma_A-k\phi^\Gamma_{A^\mathsf{c}}}{2g-2}  \Bigr\rceil +h-1,
		\end{equation*}
		where $k$ denotes the number of unlabeled vertices in $Z(e,h,A)$. Thus the result holds by the following lemma.
	\end{proof}
	
	\begin{lemma}
		Let $\Gamma\in\Ggn$ be a $2$-connected trivalent graph, that is, a connected graph without separating edges whose every vertex $v$ has valence $\textup{val}(v)=3-n(v)$, and let $\Gamma'\subset\Gamma$ be a nontrivial biconnected subgraph of genus $h$, with $\textup{val}(\Gamma')=e$. Then $\Gamma'$ has $k=e+2h-2$ unlabeled vertices.
	\end{lemma}
	\begin{proof}
		We prove the result by induction on $k$.
		For $k=0$, necessarily we have $h=0$ and $e=2$.\\
		Let $\Gamma'$ have $k+1$ unlabeled vertices. Thus there exists a vertex $v$ such that $\Gamma'=\overline{\Gamma''\sqcup v}$, where $\Gamma''$ is a biconnected subgraph with $k$ unlabeled edges. Since $v$ is trivalent, we may consider two cases:
		\begin{itemize}
			\item[val$(\Gamma'',v)=1$.] That is $g(\Gamma')=g(\Gamma'')$. By inductive hypothesis, $k=\textup{val}(\Gamma'')+2g(\Gamma'')-2$, hence $k+1=\textup{val}(\Gamma'')+2h-1$, but, since $v$ is trivalent, it must be $e=\textup{val}(\Gamma'')+1$, which proves the claim.
			\item[val$(\Gamma'',v)=2$.] That is $h=g(\Gamma'')+1$ and $e=\textup{val}(\Gamma'')-1$. Therefore, $k+1=\textup{val}(\Gamma'')+2g(\Gamma'')-1=e+2h-2$.
		\end{itemize}
	\end{proof}

	By compatibility with $\Ggn$-morphisms, the universal analogous is well defined.
	
	\begin{definition}\label{phitocuniv}
		Let $\phi=\{\phi^\Gamma\}_{\Gamma\in\Ggn}$ be a collection of a $\Ggn$-morphisms compatible numerical polarisations $\phi^\Gamma$ for each $\Gamma$ in $\Ggn$. If, for each pair of graphs $\Gamma, \Gamma'$, we have 
		\begin{equation*}
			\mEhA^\Gamma(\phi^\Gamma)=\mEhA^{\Gamma'}(\phi^{\Gamma'}),
		\end{equation*} 
		for any $(e,h,A)\in\D(\Gamma)\cap\D(\Gamma')$,
		 we define the \emph{$\phi$-induced semi-classical universal stability condition $\m(\phi)$} as
		 
		 \begin{equation*}
		 	\m(\phi):=\varprojlim_{\Gamma\in\Ggn}\m^\Gamma(\phi^\Gamma).
		 \end{equation*}
		 
		 If, furthermore, $\phi$ is a universal numerical polarisation, then $\m(\phi)$ is said to be \emph{classical}.
	\end{definition}
	
	Notice that if there exists a map $c:\Gamma\to\Gamma'$ then $\D(\Gamma')\subset D(\Gamma)$, hence a prescription for every $\Gamma\in\Ggn$ is the same as a prescription of $\phi^\Gamma$ for all the $\Gamma$ initial objects (i.e. trivalent graphs).
	
	\begin{remark}

		Clearly, not only each element $\phi^\Gamma$ of a contractions and automorphisms compatible collection $\{\phi^\Gamma\}_{\Gamma\in\Ggn}$ is compatible with $\Ggn$-morphisms, but their collection induces a family $\{\m^\Gamma\}_{\Gamma\in\Ggn}$, whose integral data $\mEhA^\Gamma$ is the same for every $\Gamma$ such that $(e,h,A)\in\D(\Gamma)$.
		As a consequence, we obtain that every classical universal stability condition is semi-classical. The converse, however, does not hold, as shown by \cite[Example~6.15]{PTgenus1} (see also Example~\ref{ex24} below).
	\end{remark}

	The following lemma allows us to check whether a $c$-stability condition on a graph $\Gamma$ is classical by checking it only on half of the elements of $\D(\Gamma)$.
	\begin{lemma}\label{phiinequlemma}
		Let $\m^\Gamma$ be a degree $d$ classical $c$-stability condition on a $2$-connected graph $\Gamma\in\Ggn$, with $g\geq2$, and let $(e,h,A)\in\D(\Gamma)$. Let $\phi^\Gamma\in\RR^n$ be a numerical polarisation on $\Gamma$ such that $\mEhA^\Gamma =\mEhA^\Gamma(\phi^\Gamma)$. Then $\m^\Gamma_{(e,g+1-e-h,A^\mathsf{c})}=\m^\Gamma_{(e,g+1-e-h,A^\mathsf{c})}(\phi^\Gamma)$.
	\end{lemma}
	\begin{proof}
		By Proposition~\ref{phitocprop}, $\mEhA^\Gamma=\mEhA^\Gamma(\phi^\Gamma)$ if and only if
		\begin{equation}\label{ceil}
			\Bigl\lceil\frac{(e+2h-2)(d+1-g)+(2g-e-2h)\phi^\Gamma_A-(e+2h-2)\phi^\Gamma_{A^\mathsf{c}}}{2g-2}  \Bigr\rceil +h-1=\mEhA^\Gamma,
		\end{equation}
		while $\m^\Gamma_{(e,g+1-e-h,A^\mathsf{c})}=\m^\Gamma_{(e,g+1-e-h,A^\mathsf{c})}(\phi^\Gamma)$ if and only if 
		\begin{equation}\label{ceil2}
			\Bigl\lceil\frac{(-e+2g-2h)(d+1-g)+(e+2h-2)\phi^\Gamma_{A^\mathsf{c}}-(2g-e-2h)\phi^\Gamma_{A}}{2g-2}  \Bigr\rceil +g-e-h=\m^\Gamma_{(e,g+1-e-h,A^\mathsf{c})}.
		\end{equation}
		Let $X$ denote the argument of the ceiling function in \eqref{ceil}. We observe that the latter equation can be written as
		\begin{equation*}
			\Bigl\lceil\frac{2g(d+1-g)}{2g-2}-\Big(X+\frac{2(d+1-g)}{2g-2}\Big)\Bigr\rceil+g-e-h=\m^\Gamma_{(e,g+1-e-h,A^\mathsf{c})},
		\end{equation*}
		that is
		\begin{equation*}
			\lceil X \rceil+h-1=d+1-e-\m^\Gamma_{(e,g+1-e-h,A^\mathsf{c})}.
		\end{equation*}
		Therefore, by property \eqref{prop1}, we conclude that equations \eqref{ceil} and \eqref{ceil2} are equivalent.
	\end{proof}
	
	Clearly, if we choose a family $\{\phi^\Gamma\}_{\Gamma\in\Ggn}$ defining a universal numerical polarisation, the same proof applied to all $(e,h,A)\in\D_{g,n}$ shows that we can check if a universal stability condition is classical on half the vertices of the vine graphs.

	\section{Fine Compactified Universal Jacobians}
	\subsection{Curves and dual graphs}
	We recall the curve-theoretic notation we will use throughout the paper. We refer to \cite{arbarello2013geometry} and \cite{pagani2023stability} for a detailed exposition.
	
	We will work over a fixed algebraically closed field $k$.
	
	A \emph{curve} over an extension $K$ of $k$ is a $\textup{Spec}(K)$-scheme $X$ proper over $K$, geometrically connected and of pure dimension $1$. A curve $X$ is \emph{nodal} if it is geometrically reduced and, when passing to an algebraic closure $\overline{K}$, its local ring at every singular point is isomorphic to $\overline{K}[[x,y]]/(xy)$.	
	A \emph{$n$-pointed curve over $K$} is a $(n+1)$-tuple $(X,p_1,\dots,p_n)$, with $X$ a curve and $p_1,\dots,p_n$ distinct smooth $K$-points of $X$.
	
	If $X$ is a $n$-pointed nodal curve, its \emph{dual graph $\Gamma(X)$} is a graph with a vertex $v$ for each irreducible component $X^v_{\overline{K}}$ of a base change of $X$ to an algebraic closure $\overline{K}$, and an edge between two vertices for each node, defined over $\overline{K}$, between the corresponding irreducible components. The dual graph $\Gamma(X)$ is weighted with $g(v)=g(X^v_{\overline{K}})$, the geometric genus of the irreducible components, and has $n$ labelled flags corresponding to the marked points. 
	
	A \emph{family of curves} over a $k$-scheme $S$ is a proper, flat morphism $\X\to S$ whose
	fibres are curves. A family of curves $\X\to S$ is a \emph{family of nodal curves} if the
	fibres over all geometric points are nodal curves.
	
	Let $\M_{g,n}$ be the moduli stack of smooth $n$-pointed curves and $\Mbargn$ its Deligne-Mumford compactification. 	
	Given a point $X\in\Mbargn$, we denote by $\Gamma(X)$ the (only) graph in $\Ggn$ isomorphic to the dual graph of $X$. Thus, we can define $c$-stability conditions on curves.
	
	\begin{definition}
		Let $[X]\in\Mbargn$. We define a \emph{degree $d$ $c$-stability condition on $X$} to be a degree $d$ $c$-stability condition $\m^{\Gamma(X)}$ on its dual graph $\Gamma(X)$.
	\end{definition}
    
	Particular relevance will be covered by \emph{vine curves}, that is curves with two irreducible components, whose dual graphs are the vine graphs. Let $(V(e,h,A),v_1)$ be a pair given by a vine graph and its vertex such that $(e,h,A)=(\textup{val}(v_1),g(v_1),A(v_1))$, we denote by $C(e,h,A)$ any vine curve whose dual graph is isomorphic to $V(e,h,A)$. 
   
    A curve $C(e,h,A)$ has two irreducible components, corresponding to the two vertices of the graph, hence we denote by $C(e,h,A)^{v_1}$ its irreducible component corresponding to the vertex $v_1$.
    When $n=0$ and $V(e,h,A)=V(g-2h+1,h,\emptyset)$, $C(g-2h+1,h,\emptyset)^{v_1}$ may identify any of the two irreducible components of the curve. This, however, doesn't represent an issue, as what follows is symmetric with respect to swapping the two components.
 
	Recall that, since \'etale specialisations of points in $\Mbargn$ induce all the morphisms in $\Ggn$ (\cite{CAVALIERI_2020}), a universal stability condition can be seen as the assignment of a stability condition on every vine curves such that they agree on common specialisations (Lemma \ref{gluelemma}). 
	
	\subsection{Fine compactified (universal) Jacobians}
	Again, the notation concerning Jacobians is taken from \cite{pagani2023stability}. We recall the main concepts we will be using in this section.
	
	Let $X$ be a nodal curve, and let $F$ be a coherent sheaf on $X$. $F$ has \emph{rank $1$} if its localisation at each generic point has length $1$, and it is \emph{torsion-free} if it has no embedded components. We say that a rank $1$ torsion-free sheaf $F$ on $X$ is \emph{simple} if its automorphism group is $\mathbb{G}_m$.
	
	We define the \emph{multidegree} of a rank $1$ torsion-free sheaf $F$ on a nodal curve $X$ by 
	
	\begin{equation*}
		\underline{\textup{deg}}(F):=(\textup{deg}(F_{\tilde{X}^v}))\in\ZZ^{V(\Gamma^X)},
	\end{equation*}
	where $F_{\tilde{X}^v}$ denotes the maximal torsion free quotient of the pullback of $F$ to the normalisation $\tilde{X}^v$ of the irreducible component $X^v$ of $X$. The \emph{total degree} of $F$ is $\textup{deg}_X(F):=\chi(F)-1+p_a(X)$, with $p_a(X)$ the arithmetic genus of $X$. We denote by Simp$^d(X)$ the space of rank $1$ torsion free simple sheaves of degree $d$ on $X$.
	
	Analogously, a \emph{family of rank $1$ torsion-free sheaves} over a family of curves $\X\to S$ is a coherent sheaf on $\X$, flat over $S$, such that its fibres over geometric points are rank $1$ torsion-free sheaves.
	
	The algebraic space $\textup{Pic}^d(\X/S)$ parametrising line bundles of relative degree $d$ embeds in the algebraic space $\textup{Simp}^d(\X/S)$ parametrising rank $1$ torsion-free simple sheaves of degree $d$ on the family of nodal curves $\X\to S$, by \cite{ALTMAN198050} and \cite{esteves}.
	
	While being locally of finite type over $S$ and satisfying the existence part of the valuative criterion for properness, as proven in \cite{esteves}, in general $\textup{Simp}^d(\X/S)$ fails to be separated and of finite type. 
	
	Given a nodal curve $X$, a \emph{degree $d$ fine compactified Jacobian} is a geometrically connected open subscheme $\overline{J}\subset\textup{Simp}^d(X)$, that is proper over $\textup{Spec}(K)$.
	
	The analogous notion is defined for families of nodal curves and in particular a \emph{degree $d$ fine compactified universal Jacobian of type $(g,n)$} is an open proper substack $\overline{\J}_{g,n}\subset\textup{Simp}^d(\overline{\C}_{g,n}/\Mbargn)$ (which is automatically connected).
	We denote its fibre over a curve $X$ by $\overline{\J}_{g,n}(X)$, dropping the subscript when it is clear from the context.

	\subsection{Correspondence with universal stability conditions}
	
    In this section, we combine the classification result for fine compactified universal Jacobians in \cite[Corollary~7.13]{pagani2023stability} and the correspondence between $V$-stabilities and fine compactified Jacobians in \cite[Theorem~1.20, Theorem~2.33]{viviani2023new}, to show that universal stability conditions bijectively correspond to fine compactified universal Jacobians. Thus, we obtain an explicit combinatorial description, thanks to which we will be able to provide examples of fine compactified universal Jacobians that are not semi-classical, i.e. whose fibres are not all classical fine compactified Jacobians.
    
    Let $\overline{\J}_{g,n}$ be a fine compactified universal Jacobian, and let $[\F]\in\overline{\J}_{g,n}$ be a family of sheaves. By \cite[Remark 7.10]{pagani2023stability}, if $X$ and $X'$ are two curves with $\Gamma(X)=\Gamma(X')$, then $\underline{\textup{deg}}(\F_{|_X})=\underline{\textup{deg}}(\F_{|_{X'})}$. In particular, for each $(e,h,A)\in\D_{g,n}$, the degree of $\F$ on the first component $C(e,h,A)^{v_1}$ of each vine curve of type $(e,h,A)$ is a well-defined integer, which we denote by $\textup{deg}_{C(e,h,A)^{v_1}}(\F)$. Thus, we can give the following definition.
    \begin{definition}\label{jtomdef}
        Let $\overline{\J}_{g,n}$ be a degree $d$ fine compactified universal Jacobian. We define the \emph{universal stability condition associated to $\overline{\J}_{g,n}$} to be $\m(\overline{\J}_{g,n})$, given by
    \begin{equation*}
        \m(\overline{\J}_{g,n})_{(e,h,A)}:=\textup{min}_{[\F]\in\overline{\J}_{g,n}}\{\textup{deg}_{C(e,h,A)^{v_1}}(\F)\},
    \end{equation*}
    for each vine graph triple $(e,h,A)\in\D_{g,n}$.
    \end{definition} 
    \begin{remark}\label{jtomremark}
    Notice that the restriction $\m(\overline{\J}_{g,n})^\Gamma_{(e,h,A)}$ of $\m(\overline{\J}_{g,n})$ to any graph $\Gamma$ coincides with the $V$-stability condition associated to the poset of orbits of \cite[Definition 2.14]{viviani2023new}, via Theorem 1.20 in loc.cit. Therefore, since clearly $\m(\overline{\J}_{g,n})^\Gamma_{(e,h,A)}=\m(\overline{\J}_{g,n})^{\Gamma'}_{(e,h,A)}$, for each triple $(e,h,A)\in\D(\Gamma)\cap\D(\Gamma')$, Lemma~\ref{gluelemma} shows that $\m(\overline{\J}_{g,n})=\varprojlim_{\Gamma\in\Ggn}\{\m(\overline{\J}_{g,n})^\Gamma\}$ is indeed a universal stability condition.   

    More specifically, the bijection provided by \cite[Theorem~1.20]{viviani2023new}, together with the compatibility with all the morphisms of $\Ggn$ induced by \'etale specialisations in $\Mbargn$ guaranteed by the above reasoning, maps $\m(\overline{\J}_{g,n})$ to the \emph{associated assignment} (\cite[Definition~7.1]{pagani2023stability}) $\sigma_{\overline{\J}_{g,n}}$ of ${\overline{\J}_{g,n}}$.
    \end{remark}

    Viceversa, we assign a fine compactified universal Jacobian to a universal stability condition in the following way.
    \begin{definition}\label{stablesheavesdef}
    Let $\F$ be a family of rank $1$ torsion free simple sheaves on the universal family $\overline{\C}_{g,n}/\Mbargn$, and let $\m$ be a universal stability condition of type $(g,n)$. We say that $\F$  is \emph{$\m$-stable} if
    the inequality 
    \begin{equation*}
        \textup{deg}_{C(e,h,A)^{v_1}}(\F)\geq \mEhA
    \end{equation*}
    is satisfied for all $(e,h,A)\in\D_{g,n}$.
    \end{definition}
    \begin{lemma-definition}\label{mtojlemma}
    Let $\m$ be a degree $d$ universal stability condition of type $(g,n)$. The \emph{fine compactified universal Jacobian associated to $\m$} as
    \begin{equation*}
        \overline{\J}_{g,n}(\m):=\{\F\in\textup{Simp}^d(\overline{\C}_{g,n}/\Mbargn): \textup{$\F$ is $\m$-stable} \}
    \end{equation*}
    is, indeed, a fine compactified universal Jacobian.
    \end{lemma-definition}
    \begin{proof}
        By Lemma~\ref{gluelemma}, $\m$ is the inverse limit of a family $\{\m^\Gamma\}_{\Gamma\in\Ggn}$. For each $\Gamma\in\Ggn$ and for each curve $X\in\Mbargn$ such that $\Gamma(X)=\Gamma$, the fibre  $\overline{\J}_{g,n}(\m)_{|_X}$ coincides with the $V$-open subset $\overline{J}_X(\m^\Gamma)$ of \cite[Lemma-Definition~2.30]{viviani2023new}, by Proposition~\ref{cstab=vstab}. Moreover, if $\sigma(\m^\Gamma)$ is the stability condition (according to \cite[Definition~4.3]{pagani2023stability}) associated to $\m^\Gamma$ by \cite[Theorem~1.20]{viviani2023new}, \cite[Theorem~2.33]{viviani2023new} states that $\overline{J}_X(\m^\Gamma)=\overline{J}_{\sigma(\m^\Gamma)}$ (\cite[Definition~6.2]{pagani2023stability}).
        
        Finally, since $\mEhA^\Gamma=\mEhA^{\Gamma'}$ for each $(e,h,A)\in\D(\Gamma)\cap\D(\Gamma')$, the family $\{\sigma(\m^\Gamma)\}_{\Gamma\in\Ggn}$ is a family of degree $d$ stability conditions on $\overline{\C}_{g,n}/\Mbargn$ (\cite[Definition~4.7]{pagani2023stability}). Therefore, by \cite[Theorem~6.3]{pagani2023stability}, we conclude that $\overline{\J}_{g,n}(\m)$ is a degree $d$ fine compactified universal Jacobian of type $(g,n)$.
    \end{proof}

    In particular, we give the following:
    
    \begin{definition}\label{classicaljacdef}
        Let $\m$ be a degree $d$ universal stability condition of type $(g,n)$. We say that the fine compactified universal Jacobian $\overline{\J}_{g,n}(\m)$ associated to $\m$ is \emph{classical} (resp. \emph{semi-classical}) if $\m$ is classical (resp. semi-classical). 

        Let $[X]\in\Mbargn$. We say that the fibre $\overline{\J}_X(\m^{\Gamma(X)})$ of $\overline{\J}_{g,n}(\m)$ over $[X]$ is \emph{classical} if the stability condition $\m^{\Gamma(X)}=\m_{|_{\D(\Gamma(X))}}$ on the dual graph $\Gamma(X)$ is classical.
    \end{definition}

    Finally, we may prove the main result  of this section, that is, we obtain a complete classification of fine compactified universal Jacobians in terms of universal stability conditions.	
	
	\begin{proof}[Proof of Theorem~\ref{mainthm}] 
 
    Since both arrows are well defined by the above discussion, we need only to show they are indeed mutual inverses.
        Let $\overline{\J}_{g,n}$ be a degree $d$ fine compactified universal Jacobian. By Remark~\ref{jtomremark}, the bijections given by \cite[Theorem~1.20]{viviani2023new} map the stability condition $\m(\overline{\J}_{g,n})$ to the associated assignment of $\overline{\J}_{g,n}$. Now, the proof of Lemma-Definition~\ref{mtojlemma}, shows that the fine compactified universal Jacobian associated to $\m(\overline{\J}_{g,n})$, is exactly the moduli space of $\sigma(\m)$-stable sheaves (according to \cite[Definition~6.2]{pagani2023stability}). Hence \cite[Corollary~7.19]{pagani2023stability} proves that $\overline{\J}_{g,n}(\m(\overline{\J}_{g,n}))=\overline{\J}_{g,n}$. Analogously, by the second part of the same corollary, we see that $\m(\overline{\J}_{g,n}(\m))=\m$, thus concluding our proof.

	\end{proof}

	\vspace{.5cm}
	
	Such a description allows us to show that, for each hyperbolic pair $(g,n)$, there exists a finite number of degree $d$ fine compactified universal Jacobians of type $(g,n)$, up to the translation action of the universal, relative degree~$0$ Picard group $\textup{PicRel}^0_{g,n}(\ZZ)$ (\cite[Definition~2.1]{Kass_2019}). If $g\geq 2$ and $n\geq 1$, the generators of $\textup{PicRel}^0_{g,n}(\ZZ)$ are:
	\begin{itemize}
		\item the boundary divisors $\O(C(1,h,A)^{v_1})$, for any $(1,h,A)\in\D_{g,n}$;
		\item the line bundles $\O(\Sigma_1-\Sigma_j)$, for $j=2,\dots,n$;
		\item $\O((2g-2)\Sigma_1)\otimes\omega_\pi^{\otimes-1}$.
	\end{itemize}
		 Here $\Sigma_j$ denotes the $j$-th section of the universal curve $\pi:\overline{\C}_{g,n}\to\overline{\M}_{g,n}$, and $\omega_\pi$ is the relative dualising sheaf.

	Recall that $\mEhA$ is the minimum degree of a line bundle in the fine compactified universal Jacobian associated to $\m$ on the genus $h$ component with markings $A$, of a vine curve with $e$ edges. Hence, the action of $\textup{PicRel}^0_{g,n}(\ZZ)$ by tensor product with the multidegree of a line bundle on all curves on the set of universal stability conditions is the translation by the degree of such line bundle. In particular, the generators of $\textup{PicRel}^0_{g,n}(\ZZ)$ act in the following way:
	\begin{align*}
		\O(C(1,h',A')^{v_1})\cdot\mEhA&=\begin{cases*}
			\mEhA-1,\text{ if } (e,h,A)=(1,h',A')\\
			\mEhA+1,\text{ if } (e,h,A)=(1,g-h',[n]\setminus A')\\
			\m_{(e,h,A)},\text{ otherwise}
		\end{cases*};\\[5pt]
		\O(\Sigma_1-\Sigma_j)\cdot \m_{(e,h,A)}&=\begin{cases*}
			\m_{(e,h,A)}+1,\text{ if } 1\in A \text{ and }j\notin A\\
			\m_{(e,h,A)}-1,\text{ if } 1\notin A \text{ and }j\in A\\
			\m_{(e,h,A)},\text{ otherwise}		\end{cases*};\\[5pt]
		\O((2g-2)\Sigma_1)\otimes\omega_\pi^{\otimes-1}\cdot \m_{(e,h,A)}&=\begin{cases*}
			\m_{(e,h,A)}+2g-2h-e,\text{ if } 1\in A\\
			\m_{(e,h,A)}-2h-e+2,\text{ if } 1\notin A
		\end{cases*}.
	\end{align*}
	This leads us to the following. 
	\begin{lemma}\label{translationlemma}
		The degree $d$ universal stability conditions $\m$ on of type $(g,n)$, with $g\geq 2$ and $n\geq 1$, such that:
		\begin{itemize}
			\item\label{boundDiv} $\m_{(1,h,A)}=\begin{cases*}
				0, \textup{ if }1\in A\\
				d, \textup{ if }1\notin A
			\end{cases*}$, for each $(1,h,A)\in\D_{g,n}$;
			\item $\m_{(2,0,\{1\})}\in\{0,\dots,2g-3\}$;
			\item $\m_{(2,0,\{i\})}=0$ for each $i\neq 1$;
		\end{itemize}
 form a full set of representatives for the action of $\textup{PicRel}^0_{g,n}(\ZZ)$ on the set of degree $d$ universal stability conditions $\m$ of type $(g,n)$.
	\end{lemma}
	\begin{proof} 	Let $\m$ be a universal stability condition. Applying $\otimes\O(C(1,h,A)^{v_1})^{\otimes\m_{(1,h,A)}}$, for each $(1,h,A)\in\D_{g,n}$ such that $1\in A$, we obtain a translation equivalent stability condition that satisfies the first requirement.
		
	Now, let us consider the latter two conditions.
	First, we want to show that $\m$ is identified by a set of differences with its elements of the form $\m_{(2,0,\{i\})}$, for $i=1,\dots,n$. The statement is clearly true for $\m_{(2,0,A)}$, where $A\neq \emptyset$. In fact, the iterate application of property \eqref{prop2}, shows that $\m_{(2,0,A)}$ is uniquely determined by 
		
		\begin{equation*}
			\m_{(2,0,A)}-\sum_{i\in A}\m_{(2,0,\{i\})}.
		\end{equation*}
		
		We observe that, by the same property, an integer $\m_{(e,h,A)}$, with $(e,h)\neq(2,0)$ and $A\neq\emptyset$, is uniquely determined by the difference $\mEhA-\m_{(2,0,A)}-\m_{(e,h,\emptyset)}$. Thus we are left to show that the result holds for triples $(e,h,\emptyset)$.
		
		Each assignment of type $\m_{(e,h,\emptyset)}$, with $(e,h)\neq(2,0)$ is determined by the values $\m_{(e,h,\emptyset)}-\m_{(e',0,\emptyset)}-\m_{(e'',0,\emptyset)}$, with $e'+e''=e+2h+2$. Notice that such $e'$ and $e''$ always exist, since $e+2h+2\leq 2g+4-e\leq 2g+2$, as $e\geq 2$. Moreover, combining properties \eqref{prop1}
		and \eqref{prop2}, we obtain the system
		
		\begin{equation*}
			\begin{cases*}
				\m_{(g+1,0,\emptyset)}+\m_{(g+1,0,[n])}=d-g,\\
				0\leq\m_{(g+1,0,[n]}-\m_{(g+1,0,\emptyset)}-\m_{(2,0,[n])}\leq1,
			\end{cases*}
		\end{equation*}
		and thus, $\m_{(g+1,0,\emptyset)}$ is determined by the difference $d-g-2\m_{(g+1,0,\emptyset)}-\m_{(2,0,[n])}$.
		Finally, for each $e\geq4$, the inequality $0\leq\m_{(e,0,\emptyset)}-\m_{(e-1,0,\emptyset)}-\m_{(3,0,\emptyset)}\leq 1$ holds, and, in particular, $0\leq\m_{(4,0,\emptyset)}-2\m_{(3,0,\emptyset)}\leq 1$. Hence, the system obtained by combining all these relations, shows that, for any $e\geq 3$, the value of $\m_{(e,0,\emptyset)}$ ultimately depends on a series of differences concerning $\m_{(2,0,[n])}$, and thus only integers of the form  $\m_{(2,0,\{i\})}$, for $i=1,\dots,n$.
		
		Now, we show that for each $(e,h,A),(e',h',A'),(e'',h'',A'')\in\D_{g,n}$ such that $A=A'\cup A'', A'\cap A''=\emptyset, \text{ } 2(h+1-(h'+h''))=e'+e''-e$, and there exists a triangle whose edges have length $e,e'$ and $e''$ respectively, and for each $\L\in\textup{PicRel}^0_{g,n}(\ZZ)$, we have $\L\cdot(\mEhA-\m_{(e',h',A')}-\m_{(e'',h'',A'')})=\mEhA-\m_{(e',h',A')}-\m_{(e'',h'',A'')}$. It is enough to show it for the elements of a base of $\textup{PicRel}^0_{g,n}(\ZZ)$:
		\begin{itemize}
			\item \begin{equation*}
				\O(\Sigma_1-\Sigma_j)\cdot(\mEhA-\m_{(e',h',A')}-\m_{(e'',h'',A'')})=\mEhA-\m_{(e',h',A')}-\m_{(e'',h'',A'')},
			\end{equation*}
			trivially, since $A=A'\cup A''$ and $A'\cap A''=\emptyset$.
			\item \begin{multline*}
				\O((2g-2)\Sigma_j)\otimes\omega_\pi^{\otimes-1}\cdot(\mEhA-\m_{(e',h',A')}-\m_{(e'',h'',A'')})=\\
				=(\mEhA-\m_{(e',h',A')}-\m_{(e'',h'',A'')})+(2(h'+h''-h-1)-(e-e'-e''))=\\=\mEhA-\m_{(e',h',A')}-\m_{(e'',h'',A'')},
			\end{multline*}
			 by the assumptions on $(e,h,A),(e',h',A'),(e'',h'',A'')$.
		\end{itemize}
		
		Therefore, we can conclude that the equivalence classes for the action of $\textup{PicRel}^0_{g,n}(\ZZ)$ depend only on the assignments $\m_{(2,0,\{i\})}$, for $i=1,\dots,n$, where the generators act as 
		\begin{equation*}
			\O(\Sigma_1-\Sigma_j)\cdot \m_{(2,0\{i\})}=\m_{(2,0\{i\})}+\delta_{1,i}-\delta_{j,i},\qquad \O(2\Sigma_1)\otimes\omega_\pi^{\otimes-1}\cdot \m_{(2,0\{i\})}=\m_{(2,0\{i\})}+(2g-2)\delta_{1,i}.
		\end{equation*}
		
		Now, we can show that any universal stability condition $\m$ is equivalent to one such that $\m_{(2,0,\{1\})}\in\{0,\dots,2g-3\}$ and $\m_{(2,0,\{i\})}=0$ for each $i\neq 1$, via the relation induced by $\textup{PicRel}^0_{g,n}(\ZZ)$.
		Applying $\bigotimes^n_{j=2}\O(\Sigma_1-\Sigma_j)^{\otimes\m_{(2,0\{j\})}}$, we see that $[\m]=[\tilde{\m}]$, with 
		\begin{equation*}
			\tilde{\m}_{(2,0,\{i\})}=\begin{cases*}
				\sum_{j=1}^{n}\m_{(2,0\{j\})},\text{ for }i=1,\\
				0,\text{ for }i\neq1.
			\end{cases*}
		\end{equation*}
		Acting by $(\O(2\Sigma_1)\otimes\omega_\pi^{\otimes-1})^{\otimes-\lfloor\frac{\sum_{j=1}^{n}\m_{(2,0\{j\})}}{2g-2}\rfloor}$, we see that $\m$ is indeed equivalent to a stability condition $\m'$ with $\m'_{(2,0,\{i\})}$ equal to the class of equivalence modulo $(2g-2)$ of $\sum_{j=1}^{n}\m_{(2,0\{j\})}$.
		
		Finally, since the class of $\m'_{(2,0,\{1\})}$ modulo $(2g-2)$ can be changed only applying a combination of the generators $\O(\Sigma_1-\Sigma_j)$, whose sum of coefficients is not a multiple of $2g-2$, we deduce that two stability conditions $\m'$ and $\m''$ such that $\m'_{(2,0,\{1\})}\neq\m''_{(2,0,\{1\})}$ and $\m'_{(2,0,\{i\})}=\m''_{(2,0,\{i\})}=0$ for each $i\neq0$ are never translation equivalent, thus the claim follows.
	\end{proof} 
    
    We can then prove the second part of
    Theorem~\ref{theoremintro}:

    \begin{proof}[Proof of Theorem~\ref{theoremintro}, 2.]
    
        Clearly, given an assignment $\{\n_A\}_{\emptyset\neq A\subset[n]}$, there exist only a finite amount of degree $d$ universal stability conditions $\m$ of type $(g,n)$ such that $\m_{(2,0,A)}=\n_A$ for each $\emptyset\neq A\subset[n]$, as they need to satisfy \eqref{prop2}. Since the assignments on the triples $(1,h,A)$ are fixed, this proves the statement for each pair $(g,n)$, with $g\geq 2$ and $n\geq 1$.

        In the case of genus $g=0$ the statement is trivial, while for $g=1$ it follows from \cite[Theorem~6.5]{PTgenus1}.
        
        By \cite[Theorem~9.8]{pagani2023stability}, there is at most one degree $d$ universal stability condition of type $(g,0)$, for any genus $g$. Thus, we conclude.
    \end{proof}

    \begin{remark}
        A further consequence of Lemma~\ref{translationlemma} and part two of Theorem~\ref{theoremintro} is that degree $d$ fine compactified universal Jacobians of type $(g,n)$ are finite up to isomorphisms that commute with the forgetful morphism to $\Mbargn$. In fact, \cite[Lemma~6.13]{Kass_2019}, with $L$ of degree $0$ and $t=0$ if $d\neq 0$, holds true with the same proof.
    \end{remark}
	
	\subsection{Comparison with classical fine compactified universal Jacobians}\label{seccomparison}
	In this section, we use our classification in terms of universal stability conditions to construct explicit examples of non (semi-)classical fine compactified universal Jacobians. Not only, we show that in general, for any $g\geq 1$ the class of fine compactified universal Jacobians is strictly larger than the class of classical fine compactified universal Jacobian described in \cite{Kass_2019}, but we also find the minimum number of markings $n$ for this to be true over $\Mbargn$. 
	
	Moreover, we are able to make a further distinction between semi-classical fine compactified universal Jacobians and \emph{all} fine compactified universal Jacobians. In particular, for genus $g=2$, with the help of computations by Rhys Wells \cite{urlFCJ}, we see that there is a range $4\leq n \leq 5$ where not all the fine compactified universal Jacobians are classical, but their restrictions over all geometric points are.
	
	In the body of \cite[Section $10$]{pagani2023stability} Question~\ref{questintro} is asked.
	A complete answer is given in Part~$1.$ of Theorem~\ref{theoremintro}, which will be proven via the examples and propositions of this section, together with the following
		
	\begin{theorem}\label{comparisonthm}
		The inclusion
		\begin{equation*}
			 \begin{Bmatrix}
			 \textup{classical fine compactified}\\ \textup{universal Jacobians of type $(g,n)$}
			 \end{Bmatrix}
			 \hookrightarrow 
			 \begin{Bmatrix}
			 	\textup{Fine compactified universal}\\ \textup{Jacobians of type $(g,n)$}
			 \end{Bmatrix}
		\end{equation*}
		is an equality if and only if $(g,n)$ satisfies any of the following:
		\begin{itemize}
			\item $g=0$;
			\item $n=0$;
			\item $g=1$ and $n\leq5$;
			\item $g=2$ and $n\leq3$;
			\item $g=3$ and $n=1$.
		\end{itemize}
	\end{theorem}
	
	For the whole section, we set 
	$$\m_{(1,h,A)}=\begin{cases*}
		0, \textup{ if }1\in A\\
		d, \textup{ if }1\notin A
	\end{cases*},\textup{ for each }(1,h,A)\in\D_{g,n},$$ for any universal stability condition $\m$ of type $(g,n)$.
	
	We start by considering the case of genus $g=2$.
	\begin{example}\label{ex26}
		Let $\m$ be the degree $3$ universal stability condition on $G_{2,6}$ defined in the following way:
		\begin{equation*}
			\begin{cases*}
				\m_{(2,0,A)}=0, \text{ for each }A\subseteq[6],\\
				\m_{(3,0,A)}=\mathbbm{1}_{\S^\mathsf{c}},\\
				\m_{(2,1,A)}=2, \text{ for each }A\subseteq[6], 
			\end{cases*}
		\end{equation*}
	where $\mathbbm{1}_{\S^\mathsf{c}}$ denotes the characteristic function of the complementary of the set $\S\subset\P([6])$, such that $A\in\S$ if and only if $|A|\leq 2$ or \[A=\{1,3,5\},\{1,2,3\}\{1,5,6\}\{1,2,4\}\{1,3,6\}\{3,4,5\}\{3,4,6\}\{2,5,6\}\{2,3,5\}\{2,3,6\}.\]
	
	Since $A\in\S$ if and only if $\A^\mathsf{c}\notin\S$, we deduce that $\m$ is indeed a universal stability condition.
	
	Let $\Gamma$ be the dual graph in figure \ref{26}, where all the components have genus $0$.
	
	
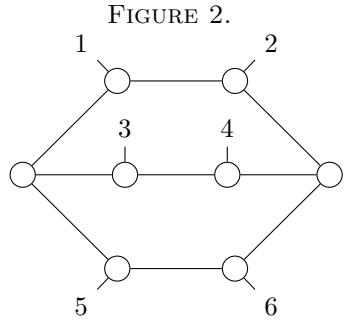
\begin{figure}[H]
		\caption{}
		\label{26}
		\centering
		\begin{tikzpicture}
			\node(0)[shape=circle,draw]{};
			\node(3)[right = of 0][shape=circle,draw]{};
			\node(4)[right = of 3][shape=circle,draw]{};
			\node(7)[right = of 4][shape=circle,draw]{};
			\node(1)[above right = of 0][shape=circle,draw]{};
			\node(2)[above left = of  7][shape=circle,draw]{};
			\node(5)[below right = of 0][shape=circle,draw]{};
			\node(6)[below left = of  7][shape=circle,draw]{};
			\node(l1)[above left =.2 of 1]{1};
			\node(l2)[above right =.2 of 2]{2};
			\node(l3)[above =.2 of 3]{3};
			\node(l4)[above =.2 of 4]{4};
			\node(l5)[below left =.2 of 5]{5};
			\node(l6)[below right =.2 of 6]{6};
			\draw (1) to (l1);
			\draw (2) to (l2);
			\draw (3) to (l3);
			\draw (4) to (l4);
			\draw (5) to (l5);
			\draw (6) to (l6);
			
			\draw (0) to (3);
			\draw (0) to (1);
			\draw (0) to (5);
			\draw (4) to (3);
			\draw (5) to (6);
			\draw (1) to (2);
			\draw (7) to (6);
			\draw (7) to (2);
			\draw (7) to (4);
			
		\end{tikzpicture}
	\end{figure}
	Let $\m^\Gamma:=\m_{|_{\D(\Gamma)}}$ be the stability condition on $\Gamma$ obtained by restricting $\m$ to the values $\mEhA$ such that $(e,h,A)\in\D(\Gamma)$. By Proposition~\ref{phitocprop}, for a numerical polarisation $\phi^\Gamma$ on $\Gamma$ to be such that $\m^\Gamma=\m^\Gamma(\phi^\Gamma)$, it needs to satisfy 
	
	\begin{equation*}
		\begin{cases*}
			\Bigl\lceil \frac{2+\phi^\Gamma_{\{1,2,3\}}-\phi^\Gamma_{\{4,5,6\}}}{2}\Bigr\rceil=1,\\
			\Bigl\lceil \frac{2+\phi^\Gamma_{\{1,5,6\}}-\phi^\Gamma_{\{2,3,4\}}}{2}\Bigr\rceil=1,\\
			\Bigl\lceil \frac{2+\phi^\Gamma_{\{3,4,5\}}-\phi^\Gamma_{\{1,2,6\}}}{2}\Bigr\rceil=1,\\
		\end{cases*}
	\end{equation*}
	and
	\begin{equation*}
		\begin{cases*}
			\Bigl\lceil \frac{2+\phi^\Gamma_{\{1,2,5\}}-\phi^\Gamma_{\{3,4,6\}}}{2}\Bigr\rceil=2,\\
			\Bigl\lceil \frac{2+\phi^\Gamma_{\{1,3,4\}}-\phi^\Gamma_{\{2,5,6\}}}{2}\Bigr\rceil=2,\\
			\Bigl\lceil \frac{2+\phi^\Gamma_{\{3,5,6\}}-\phi^\Gamma_{\{1,2,4\}}}{2}\Bigr\rceil=2.\\
		\end{cases*}
	\end{equation*}
	This implies
	\begin{equation*}
		\begin{cases*}
			\phi^\Gamma_{\{1\}}-\phi^\Gamma_{\{2\}}+\phi^\Gamma_{\{3\}}-\phi^\Gamma_{\{4\}}+\phi^\Gamma_{\{5\}}-\phi^\Gamma_{\{6\}}<0,\\
			\phi^\Gamma_{\{1\}}-\phi^\Gamma_{\{2\}}+\phi^\Gamma_{\{3\}}-\phi^\Gamma_{\{4\}}+\phi^\Gamma_{\{5\}}-\phi^\Gamma_{\{6\}}>0,
		\end{cases*}
	\end{equation*}
	which is clearly impossible. Therefore, this example shows that there exists a non semi-classical fine compactified universal Jacobian.
	
	Moreover, one can extend the stability condition $\m$ to $\m'$ over $\overline{\M}_{2,n}$ for any $n\geq6$, by imposing $\mEhA':=\m_{(e,h,A\cap[6])}$. Thus, the same result holds for any $n\geq6$.
	\end{example}
	
	\begin{remark}\label{rhys25}
		By work of Rhys Wells \cite{urlFCJ}, this example is minimal: given any $n$-pointed curve $X$ of genus $2$ with $n\leq 5$, every fine compactified Jacobians of $X$ is classical.
	\end{remark}

	\begin{example}\label{ex24}
		Let $\m$ be the degree $0$ universal stability condition on $G_{2,4}$ defined by the following table, where the rows are indexed by pairs $(e,h)$ and the columns by sets $A$.
		
			\begin{table}[ht!]
			\centering
			\begin{tabular}{||c|| c c c c c c c c c c c c c c c c||} 
				\hline
				& $\emptyset$ & $1$ & $2$ & $3$  & $4$ & $12$ & $13$ & $14$ & $23$ & $24$ & $34$ & $123$ & $124$ & $134$ & $234$ & $1234$\\ [0.5ex] 
				\hline\hline
				$(2,0)$& $/$ & $1$ & $0$ & $0$ & $0$ & $1$ & $1$ & $2$ & $0$ & $0$ & $0$ & $2$ & $2$ & $2$ & $0$ & $2$ \\ 
				[1ex] 
				\hline
				$(3,0)$ & $-2$ & $-1$ & $-2$ & $-2$ & $-2$ & $-1$ & $-1$ & $0$ & $-2$ & $-1$ & $-1$ & $0$ & $0$ & $0$ & $-1$ & $0$\\
				[1ex] 
				\hline
				$(2,1)$ & $-3$ & $-1$ & $-3$ & $-3$ & $-3$ & $-1$ & $-1$ & $-1$ & $-3$ & $-2$ & $-2$ & $-1$ & $-1$ & $-1$ & $-2$ & $/$\\
				[1ex] 
				\hline
			\end{tabular}
		\end{table}
		One can check that it satisfies the properties of Definition~\ref{cstabuniv}. Suppose there exists a universal numerical polarisation $\phi$ such that $\m=\m(\phi)$. It would satisfy
		\begin{equation}\label{eq24}
			\begin{cases*}
				\Bigl\lceil\frac{-1+\phi_{\{1,2\}}-\phi_{\{3,4\}}}{2}\Bigr\rceil=0,\\
				\Bigl\lceil\frac{-1+\phi_{\{1,3\}}-\phi_{\{2,4\}}}{2}\Bigr\rceil=0,\\
				\lceil\phi_{\{2,3,4\}}\rceil=1,\\
				\lceil\phi_{\{1,2,3\}}\rceil=3,
			\end{cases*}
		\end{equation}
		which altogether imply
		\begin{equation*}
			\begin{cases*}
				\phi_{\{1,2,3\}}<2,\\
				\phi_{\{1,2,3\}}>2.
			\end{cases*}
		\end{equation*}
		So, we can conclude that $\m$ is not a classical universal stability condition. However, by remark \ref{rhys25}, or simply by noticing that no single graph $\Gamma$ in $G_{2,4}$ is such that $(3,0,\{1,2\}),(3,0,\{1,3\})$, $(2,0,\{2,3,4\}),(2,0,\{1,2,3\})\in\D(\Gamma)$ and that no contradictions arise when removing any of the conditions in \eqref{eq24}, we deduce that $\m$ restricts to a classical stability condition on the dual graph of every geometric point of $\overline{\M}_{2,4}$, hence, it is semi-classical. 
		
		Moreover, as in the previous example, it is clear that this can be extended to any $n\geq4$.
	\end{example}
	
	\begin{proposition}\label{prop23}
		Every degree $d$ fine compactified universal Jacobian over $\overline{\M}_{2,n}$ is classical for $n\leq 3$.
	\end{proposition}
	\begin{proof}
		A universal stability condition $\m$ on $G_{2,n}$ is equivalent to a "universal pair" $(f^2,f^3)$ appearing in the "f2\_f3\_txt\_poly\_dims" repository in \cite{urlFCJ}, which contains computations by Rhys Wells based on a previous version of this paper. There, it is shown that, for $1\leq n\leq3$, if a universal stability condition $\m$ is such that $\m_{(2,0,\{1\})}\in\{0,1\}$ and $\m_{(2,0,\{i\})}=0$ for each $i\neq 1$, then $\m$ is classical. Thus, the claim follows from Lemma \ref{translationlemma}.
	\end{proof}

	Combining Examples \ref{ex26} and \ref{ex24}, Remark \ref{rhys25} and Proposition \ref{prop23}, we can show the following result, which describes completely the inclusion of fine (semi-)classical compactified Jacobians into the class of all fine compactified universal Jacobians over $\overline{\M}_{2,n}$.
	
	\begin{corollary}\label{g2comparison}
		The inclusions
		\begin{equation*}
			\begin{Bmatrix}
				\textup{classical fine}\\ \textup{compactified universal}\\
				\textup{Jacobians of type $(2,n)$}
			\end{Bmatrix}
			\hookrightarrow 
			\begin{Bmatrix}
				\textup{semi-classical fine}\\ \textup{compactified universal}\\
				\textup{Jacobians of type $(2,n)$}
			\end{Bmatrix}
			\hookrightarrow
			\begin{Bmatrix}
				\textup{Fine compactified}\\ \textup{universal Jacobians}\\
				\textup{of type $(2,n)$}
			\end{Bmatrix}
		\end{equation*}
		 are strict in general. Moreover, the first inclusion is an equality for $n\leq3$, while the second is an equality for $n\leq5$.
	\end{corollary}
	
	We now treat the case of genus $g=3$ and show that there exist non semi-classical fine compactified universal Jacobians $\overline{\J}(\m)\subset \textup{Simp}^d(\overline{\C}_{3,n}/\overline{\M}_{3,n})$ if and only if $n\geq 2$.	
	
	\begin{proposition}\label{prop31}
		Every degree $d$ fine $c$-compactified universal Jacobian over $\overline{\mathcal{M}}_{3,1}$ is classical.
	\end{proposition}
	\begin{proof}
		We want to show that for each universal stability condition $\m$ identifying a fine compactified universal Jacobian, there exists a universal numerical polarisation $\phi$ such that $\m=\m(\phi)$.
		
		First, we notice that, for any $4$-tuple of integers $(A,B_0,B_1,C)$, there is at most one stability condition $\m$ such that  $(\m_{(2,0,\{1\})}, \m_{(3,0,\emptyset)},\m_{(3,0,\{1\})},\m_{(4,0,\emptyset)})=(A,B_0,B_1,C)$. In fact, by both properties in definition \ref{cstabuniv},
		\begin{equation}\label{m4}
			\m_{(4,0,\{1\})}=d+1-4-\m_{(4,0,\emptyset)}\in\{\m_{(2,0,\{1\})}+\m_{(4,0,\emptyset)},\m_{(2,0,\{1\})}+\m_{(4,0,\emptyset)}+1\},
		\end{equation}
		
		and, similarly
		\begin{equation*}
			\m_{(2,1,\{1\})}=d+1-2-\m_{(2,1,\emptyset)}\in\{\m_{(2,0,\{1\})}+\m_{(2,1,\emptyset)},\m_{(2,0,\{1\})}+\m_{(2,1,\emptyset)}+1\}.
		\end{equation*}

		Thus, $C$ and $\m_{(2,1,\emptyset)}$ are integers such that
		\begin{align*}
			\frac{d-4-A}{2}&\leq C\leq \frac{d-3-A}{2}\\
			\frac{d-2-A}{2}&\leq \m_{(2,1,\emptyset)}\leq \frac{d-1-A}{2}.
		\end{align*}
		Therefore,
		\begin{equation*}
			\begin{cases*}
				(C,\m_{(2,1,\emptyset)})=(\frac{d-4-A}{2},\frac{d-2-A}{2}),\textup{ if } d-A\equiv0\mod 2,\\
				(C,\m_{(2,1,\emptyset)})=(\frac{d-3-A}{2},\frac{d-1-A}{2}),\textup{ if } d-A\equiv1\mod 2
			\end{cases*}
		\end{equation*}
		Either way, $\m_{(2,1,\emptyset)}=C+1$. Since $\mEhA$ is determined by complementarity via property \eqref{prop1}, for all the other triples $(e,h,A)\in\D_{g,n}$, our first claim is proven.
		
		It remains to show that, given any stability condition $\m$ such that $\m_{(2,0,\{1\})}=A$, $ \m_{(3,0,\emptyset)}=B_0,$ $\m_{(3,0,\{1\})}=B_1$ and $\m_{(4,0,\emptyset)}=C$, there exists a universal numerical polarisation $\phi$ such that $\m=\m(\phi)$. Since the inequalities induced by $\m_{(2,1,\emptyset)}$ and $\m_{(4,0,\emptyset)}$ are the same, by lemma \ref{phiinequlemma}, this amounts to showing that there exists a solution to the following system:
		\begin{equation}\label{phisyst}
			\begin{cases*}
				A<x<A+1\\
				d-4B_0-6<x<d-4B_0-2\\
				\frac{4B_1+2-d}{3}<x<\frac{4B_1+6-d}{3}\\
				d-2C-4<x<d-2C-2				
			\end{cases*}.
		\end{equation}
		By relation \eqref{m4}, and $2\m_{(3,0,\emptyset)}\leq\m_{(4,0,\emptyset)}\leq2\m_{(3,0,\emptyset)}+1$ in definition \ref{cstabuniv}, we obtain 
		\begin{equation*}
			\begin{cases*}
				(C,B_0)=(\frac{d-A-4}{2},\frac{d-A-4}{4}), \textup{ if } d-A\equiv0\mod4,\\
				(C,B_0)=(\frac{d-A-3}{2},\frac{d-A-5}{4}), \textup{ if } d-A\equiv1\mod4,\\
				(C,B_0)=(\frac{d-A-4}{2},\frac{d-A-6}{4}), \textup{ if } d-A\equiv2\mod4,\\
				(C,B_0)=(\frac{d-A-3}{2},\frac{d-A-3}{4}), \textup{ if } d-A\equiv3\mod4
			\end{cases*}.
		\end{equation*}
		In general, $B_1\in \{B_0+A,B_0+A+1\}$, however, using $\m_{(3,0,\emptyset)}+\m_{(3,0,\{1\})}\leq\m_{(4,0,\{1\})}\leq\m_{(3,0,\emptyset)}+\m_{(3,0,\{1\})}+1$, we can exclude $B_1=B_0+A$ if $d-A\equiv2\mod 4$ and $B_1=B_0+A+1$ if $d-A\equiv3\mod 4$. Hence we are left with $6$ cases, of which we are considering explicitly the first one only, as they are all completely analogous.
		
		Let $d-A\equiv0\mod 4$, so that 	$(C,B_0)=(\frac{d-A-4}{2},\frac{d-A-4}{4})$, and let $B_1=B_0+A$. Then system \eqref{phisyst} becomes
		\begin{equation*}
			\begin{cases*}
				A<x<A+1,\\
				\frac{3A-2}{3}<x<\frac{3A+2}{3},\\
				A-2<x<A+2,\\
				A<x<A+2,
			\end{cases*}
		\end{equation*}
		which admits a solution for each $A\in\ZZ$.
	\end{proof}
	
	\begin{example}\label{ex32}
		Let $\m$ be the degree $6$ universal stability condition for a compactified Jacobian over $\overline{\M}_{3,2}$ defined by the following table: 
		
		\begin{table}[H]
			\centering
			\begin{tabular}{||c|| c c c c c c ||} 
				\hline
				& $(2,0)$ & $(3,0)$ & $(4,0)$ & $(2,1)$  & $(3,1)$ & $(2,2)$\\ [0.5ex] 
				\hline\hline
				$\emptyset$ & $/$ & $0$ & $1$ & $2$ & $3$ & $4$\\ 
				[1ex] 
				\hline
				$1$& $0$ & $1$ & $2$ & $2$ & $3$ & $5$\\
				[1ex] 
				\hline
				$2$& $0$ & $1$ & $1$ & $3$ & $3$ & $5$\\
				[1ex] 
				\hline
				$12$& $1$ & $1$ & $2$ & $3$ & $4$ & $/$\\
				[1ex] 
				\hline
			\end{tabular}
		\end{table}
  
		Let $\Gamma$ be the dual graph of a maximally degenerate curve in $\overline{\M}_{3,2}$ in figure \ref{32}.
  
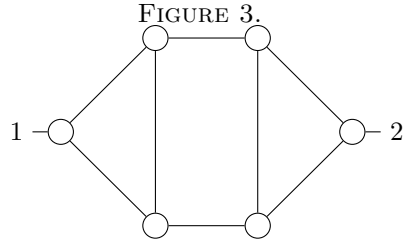
		\begin{figure}[ht!]
			\caption{}
			\label{32}
			\centering
			\begin{tikzpicture}
				\node(0)[shape=circle,draw]{};
				\node(1)[above right = of 0][shape=circle,draw]{};
				\node(2)[right = of  1][shape=circle,draw]{};
				\node(7)[below right = of 2][shape=circle,draw]{};
				\node(5)[below right = of 0][shape=circle,draw]{};
				\node(6)[below left = of  7][shape=circle,draw]{};
				\node(l1)[left =.2 of 0]{1};
				\node(l2)[right =.2 of 7]{2};

				\draw (0) to (l1);
				\draw (7) to (l2);

				\draw (0) to (1);

				\draw (5) to (6);
				\draw (1) to (2);
				\draw (7) to (2);
				
				\draw (1) to (5);
				\draw (2) to (6);
				\draw (0) to (5);
				\draw (7) to (6);
				
			\end{tikzpicture}
		\end{figure}
  
		Let $\m^\Gamma:=\m_{|_{\D(\Gamma)}}$ be the $c$-stability condition on $\Gamma$ obtained by restricting $\m$ to the values $\mEhA$ such that $(e,h,A)\in\D(\Gamma)$. By Proposition~\ref{phitocprop}, for a numerical polarisation $\phi^\Gamma$ on $\Gamma$ to be such that $\m^\Gamma=\m^\Gamma(\phi^\Gamma)$, it needs to satisfy
		\begin{equation*}
			\begin{cases*}
				\Bigl\lceil\frac{8+2\phi^\Gamma_{\{1\}}-2\phi^\Gamma_{\{2\}}}{4}\Bigr\rceil-1=2;\\
				\Bigl\lceil\frac{8+2\phi^\Gamma_{\{1\}}-2\phi^\Gamma_{\{2\}}}{4}\Bigr\rceil=2;
			\end{cases*}
		\end{equation*}
		which doesn't admit solution. Therefore $\m$ defines a fine non semi-classical compactified universal Jacobian. In particular, there exists no universal numerical polarisation $\phi$ such that $\m=\m(\phi)$.
		
		Analogously to the previous examples, this can be trivially extended to any $n\geq2$.
	\end{example}
	
	Finally, the following example shows that if the genus is at least $4$, then universal Jacobians that are non semi-classical always exist as long as there is at least $1$ marked point.
	
	\begin{example}\label{exg1}
		For any $g\geq4$, we show that there exists a fine compactified universal Jacobian over $\overline{\M}_{g,1}$ whose fibres are not all classically polarised. 
		
		Let $\m$ be the degree $g+1$ universal stability condition of type $(g,1)$ defined as follows:
		\begin{equation*}
			\mEhA=
			\begin{cases*}
				0, \text{ if } (e,h,A)=(2,0,\{1\}) \text{ or }(e,h,A)=(3,0,\{1\}),\\
				1, \text{ if } (e,h,A)=(2,1,\{1\}),\\
				h+1, \text{ if } h\geq g-2,\\
				h+\delta_{A,\{1\}}, \text{ otherwise},
			\end{cases*}
		\end{equation*}
		where $\delta_{A,\{1\}}$ denotes Kronecker's delta. We can check it satisfies both axioms of definition \ref{cstabuniv}.
		\begin{itemize}
			\item [i.] Since $g\geq e-1+h$, necessarily if $h\geq g-1$, then $e\leq 3$. Hence, we need to check only $4$ separate cases:
			\begin{align*}
				\m_{(2,0,\{1\})}+\m_{(2,g-1,\emptyset)}&=g=d+1-e,\\
				\m_{(3,0,\{1\})}+\m_{(3,g-2,\emptyset)}&=g-1=d+1-e,\\
				\m_{(2,1,\{1\})}+\m_{(2,g-2,\emptyset)}&=g=d+1-e,
			\end{align*}
			while, for each other $(e,h,A)\in\D_{g,n}$, 
			\begin{equation*}
				\mEhA+\m_{(e,g+1-e-h,A^\mathsf{c})}=g+1-e+\delta_{A,\{1\}}+\delta_{A^\mathsf{c},\{1\}}=d+1-e.
			\end{equation*}
			\item[ii.] Let $(e,h,A)\in\D_{g,n}$, and let $(e',h',A'),(e'',h'',A'')\in\D_{g,n}$ such that $A=A'\cup A'', A'\cap A''=\emptyset, \text{ } 2(h+1-(h'+h''))=e'+e''-e$, and there exists a triangle whose edges have length $e,e'$ and $e''$ respectively.
			\begin{itemize}
				\item Let $h=g-1$ Then $A=\emptyset$ and $e=2$. If $h'\geq g-2$, it must be $h''<g-2$, thus 
				\begin{equation*}
					0\leq \mEhA-h-(\m_{(e',h',A')}-h'+\m_{(e'',h'',A'')}-h'')\leq 1.
				\end{equation*}
				\item Let $h=g-2$. Then either $e=2$ or $e=3$. In the former case, $h',h''<g-2$, hence $0\leq\m_{(e',h',A')}-h'\leq\delta_{A',\{1\}}$ and $0\leq\m_{(e'',h'',A'')}-h''\leq\delta_{A'',\{1\}}$. Since $A'$ and $A''$ can not both be $\{1\}$, the property follows from $m-h=1$. 
				
				In the latter, the only different situation occurs when $h'=g-2$. Thus, either $(e'',h'',A'')=(2,0,1)$ or $(e'',h'',A'')=(3,0,1)$. Either way the property is satisfied as $\mEhA-h=1$, $\m_{(e',h',A')}-h'=1$ and $\m_{(e'',h'',A'')}-h''=0$.
				\item Let $h\leq g-3$. We can distinguish $4$ cases. 
				\begin{itemize}
					\item If $(e,h,A)=(2,0,\{1\})$, there is nothing to prove.
					\item If $(e,h,A)=(3,0,\{1\})$, then $(e',h',A')=(3,0,\emptyset)$ and $(e'',h'',A'')=(2,0,\{1\})$ (or viceversa). Thus $\mEhA-h-(\m_{(e',h',A')}-h'+\m_{(e'',h'',A'')}-h'')=0$.
					\item If $(e,h,A)=(2,1,\{1\})$, then $(e',h',A')=(3,0,\emptyset)$ and $(e'',h'',A'')=(3,0,\{1\})$ (or viceversa) or $(e',h',A')=(2,1,\emptyset)$ and $(e'',h'',A'')=(2,0,\{1\})$ (or viceversa). Thus $\mEhA-h-(\m_{(e',h',A')}-h'+\m_{(e'',h'',A'')}-h'')=0$.
					\item For each other $(e,h,A)\in\D_{g,n}$, we have $\mEhA-h=\delta_{A,\{1\}}$, $0\leq\m_{(e',h',A')}-h'\leq\delta_{A',\{1\}}$ and $0\leq\m_{(e'',h'',A'')}-h''\leq\delta_{A'',\{1\}}$. Since $A'\cup A''=A$ and $A'\cap A''=\emptyset$, the property is satisfied.
				\end{itemize}
			\end{itemize}
		\end{itemize}
		Therefore, by Theorem~\ref{mainthm}, there exists a fine compactified universal Jacobian $\overline{\J}(\m)$ over $\overline{\M}_{g,1}$ associated to $\m$.
		
		Let $\Gamma$ be the trivalent genus $g$ dual graph in figure \ref{g1}, where all vertices have genus $0$, and let $\m^\Gamma:=\m_{|_{\D(\Gamma)}}$ be the restriction of $\m$ to $\Gamma$.

		\begin{figure}[ht!]
			\caption{}
			\label{g1}
			\centering
			\begin{tikzpicture}
				\node(0)[shape=circle,draw]{};
				\node(1)[above right = of 0][shape=circle,draw]{};
				\node(2)[right = of  1][shape=circle,draw]{};
				\node(5)[below right = of 0][shape=circle,draw]{};
				\node(6)[right = of  5][shape=circle,draw]{};
				\node(l1)[left =.2 of 0]{1};
				\node(4)[right = of  	2][shape=circle,draw]{};
				\node(3)[right = of  6][shape=circle,draw]{};

                \path[-,draw]
					(4) edge[looseness=0.8, out=-55, in=55] (3);

				\draw (0) to (l1);
				
				\draw (0) to (1);
				
				\draw (5) to (6);
				\draw (1) to (2);
				\draw (1) to (5);
				\draw (2) to (6);
				\draw (0) to (5);
				\draw (4) to (2) [dashed];
				\draw (3) to (6) [dashed];
				\draw (3) to (4);
				
			\end{tikzpicture}
		\end{figure}
		By Proposition~\ref{phitocprop}, if $\m^\Gamma=\m^\Gamma(\phi^\Gamma)$ for any numerical polarisation $\phi^\Gamma$ on $\Gamma$, such a $\phi^\Gamma$ should satisfy 
		\begin{equation*}
			\begin{cases*}
				\Bigl\lceil\frac{4+(2g-4)\phi^\Gamma_{\{1\}}}{2g-2}\Bigr\rceil-1=\m_{(4,0,\{1\})}=1;\\
				\Bigl\lceil\frac{4+(2g-4)\phi^\Gamma_{\{1\}}}{2g-2}\Bigr\rceil=\m_{(2,1,\{1\})}=1;
			\end{cases*}
		\end{equation*}
		which has no solution. Therefore $\m$ defines a fine non semi-classical compactified universal Jacobian. In particular, there exists no universal numerical polarisation $\phi$ such that $\m=\m(\phi)$.
	\end{example}

	\begin{proof}[Proof of Theorem \ref{comparisonthm}]
		If and only if any of the conditions stated in the theorem holds, the inclusion is in fact an equality. We show this case by case:
		\begin{itemize}
			\item $g=0$. The statement is trivial;
			\item $g=1$. It follows from the combination of Remark $6.13$ and Example $6.15$ of \cite{PTgenus1};
			\item $g=2$ and $n\leq 3$. It is a special case of Corollary \ref{g2comparison};
			\item $g=3$ and $n=1$. It is a combination of Proposition \ref{prop31} and Example \ref{ex32};			
			\item $n=0$. The "if" statement is proven in \cite[Section 9.3]{pagani2023stability}, while Example \ref{exg1} proves the other direction.
		\end{itemize}
	\end{proof}

        Furthermore, we can give the following proof, thus answering Question~\ref{questintro}, originally posed by Pagani and Tommasi.

        \begin{proof}[Proof of Theorem~\ref{theoremintro}, Part~1.]
        
        For $g\neq 2$, it is the same as the proof of Theorem~\ref{comparisonthm}. The inclusion is strict for $g=2$ if and only if $n\leq 5$ by the combination of Example~\ref{ex26} and Remark~\ref{rhys25}.
        \end{proof}
	
    \begin{remark}
        Let $\Gamma$ be a graph. In \cite[Question~8.4]{pagani2023stability}, the authors asked whether all  stability conditions $\sigma_\Gamma$ on $\Gamma$ are induced by a numerical polarisation $\phi^\Gamma$. Example~1.27 of \cite{viviani2023new} shows that this is not the case. 

        Since $c$-stability conditions on a graph $\Gamma$ coincide with $V$-stabilities (Proposition~\ref{cstab=vstab}) that are $\Ggn$-morphisms compatible, we can improve this result by saying that the question has a negative answer for some curve in $\overline{\mathcal{M}}_{g,n}$ whenever the inclusion of Theorem~\ref{theoremintro} is strict.
    \end{remark}

	\bibliographystyle{alpha}	
	\bibliography{bibtex}

    \Address

\end{document}